\documentclass{article}
\usepackage[utf8]{inputenc}
\usepackage{tikz}
\usepackage{amsmath}
\usepackage{amssymb}
\usepackage{verbatim}
\usepackage{amsthm}
\usepackage{comment}
\usepackage{pifont}
\usepackage{amsrefs}
\usepackage{quiver}
\usepackage{tikz-cd}
\usepackage{enumitem}
\usepackage{comment}
\usepackage[skip=10pt plus1pt, indent=0pt]{parskip}

\makeatletter
\newcommand{\mylabel}[2]{#2\def\@currentlabel{#2}\label{#1}}
\makeatother

\newcommand{\Weil}{{\ensuremath{\mathbb W \hspace{-.02cm} \mathrm{eil}}}}

\newcommand{\Fun}{\ensuremath{\operatorname{Fun}}}
\newcommand{\Card}{\ensuremath{\mathsf{Card}}}
\newcommand{\Aut}{\ensuremath{\operatorname{Aut}}}
\newcommand{\Poly}{\ensuremath{\mathbb P \hspace{-.02cm} \mathrm{oly}}}
\newcommand{\Set}{\ensuremath{\mathbb S \hspace{-.07cm} \operatorname{et}}}
\newcommand{\Ring}{\ensuremath{\mathbb R \hspace{-.07cm} \operatorname{ing}}}
\newcommand{\CRing}{\ensuremath{\mathbb C \hspace{-.07cm} \operatorname{Ring}}}
\newcommand{\Grp}{\ensuremath{\mathbb G \hspace{-.07cm} \operatorname{rp}}}
\newcommand{\FinGrp}{\ensuremath{\mathbb F \hspace{-.07cm} \operatorname{inGrp}}}
\newcommand{\AbGrp}{\ensuremath{\mathbb A \hspace{-.07cm} \operatorname{bGrp}}}
\newcommand{\Alg}{\ensuremath{\mathbb A \hspace{-.07cm} \operatorname{lg}}}
\newcommand{\Obj}{\ensuremath{\operatorname{Obj}}}
\newcommand{\Grph}{\mathbb G \hspace{-.07cm} \operatorname{rph}}
\newcommand{\Hom}{\mathrm{Hom}}
\newcommand{\SmMan}{\ensuremath{\mathbb S \mathrm{mMan}}}
\newcommand{\FinSet}{{\ensuremath{\mathbb F \mathrm{inSet}}}}
\newcommand{\Top}{{\ensuremath{\mathbb T \mathrm{op}}}}
\newcommand{\CW}{{\ensuremath{\mathbb C \mathrm{W}}}}
\newcommand{\Mod}{\ensuremath{\mathbb M \mathrm{od}}}

\newcommand{\lcm}{\mathrm{lcm}}
\newcommand{\rk}{\mathrm{rk}}
\newcommand{\Vect}{\mathbb V\mathrm{ect}}
\newcommand{\FinVect}{\mathbb F\mathrm{inVect}}
\newcommand{\Dim}{\mathrm{Dim}}

\newtheorem{theorem}{Theorem}[section]
\newtheorem{definition}[theorem]{Definition}
\newtheorem{remark}[theorem]{Remark}
\newtheorem{examples}[theorem]{Examples}
\newtheorem{corollary}[theorem]{Corollary}
\newtheorem{example}[theorem]{Example}
\newtheorem{lemma}[theorem]{Lemma}

\newtheorem{proposition}[theorem]{Proposition}

\newcommand{\kristine}[1]{{\color{red}[[\textbf{Kristine says: }#1]]}}
\newcommand{\florian}[1]{{\color{teal}[[\textbf{Florian says: }#1]]}}

\usepackage[english]{babel}

\usepackage[letterpaper,top=2cm,bottom=2cm,left=3cm,right=3cm,marginparwidth=1.75cm]{geometry}

\usepackage{amsmath}
\usepackage{graphicx}
\usepackage[colorlinks=true, allcolors=blue]{hyperref}
\title{The dimension of the tangent bundle and the universality of the vertical lift}
\author{Florian Schwarz}
\begin{document}
    
\maketitle
\begin{abstract}
This paper explores a new perspective on the universality of the vertical lift in tangent categories by presenting a categorification of the dimension of smooth manifolds.  The universality of the vertical lift is a key part of the axioms of a tangent category as presented in \cite{Cockett2014DifferentialST}. The categorical dimension presented in this paper provides insight into the nature of this property.  The main result is Theorem \ref{thm:dimension_result_weak}, showing that if it exists, the dimension of the tangent bundle must fulfill an equation relating the dimension of the tangent bundle to the dimension of the base.  In particular, when the dimension function is a strong tangent dimension, Theorem \ref{thm:dimension_argument_strong} shows that the dimension of the tangent bundles is either twice the dimension of the base, or equal to the dimension of the base.  Many examples of dimension functions are provided to demonstrate the utility of the definition.  In particular, a consequence of Theorem \ref{thm:dimension_result_weak} is that there are limitations on which functors may be tangent bundle endofunctors for a category.  We show that this means that there are no non-trivial tangent structures on sets, as an example.
\end{abstract}
\tableofcontents{}
\section{Introduction}
The dimension is one of the most foundational properties characterizing a smooth manifolds. While the dimension does not uniquely characterize a manifold, it is a number that is invariant under diffeomorphisms and hence it effectively differentiates manifolds whose dimensions are not equal. The dimension is fulfills many desirable properties beyond diffeomorphism invariance: it is additive under products and it is doubled when forming the tangent bundle. In this work, we generalize the concept of a dimension and investigate the consequences in tangent categories.

\textbf{Tangent categories} categorically generalize the tangent bundle of smooth manifolds, encoding geometric structures in settings which may not traditionally be considered to be part of the field of geometry.
Tangent categories were first discovered by Ji\v{r}\'i Rosický in 1984 \cite{rosicky} and then rediscoved and reinterpreted by Robin Cockett and Geoff Cruttwell in 2012 \cite{Cockett2014DifferentialST}. The main feature of a tangent category is an endofunctor which produces an analogue of the tangent bundle, along with natural transformations which correspond to the usual projection from the bundle to the base space.  There are several other important structures involved in the definition, these are laid out in Definition \ref{def:tangent_cat}.


An important but subtle property that tangent categories are required to fulfill, according to their definition, is the universality of the vertical lift which states that
\[\begin{tikzcd}
	{T_2M} && {T^2M} \\
	M && TM
	\arrow["\nu", from=1-1, to=1-3]
	\arrow["{p_M \circ \pi_0}"', from=1-1, to=2-1]
	\arrow["\lrcorner"{anchor=center, pos=0.125}, draw=none, from=1-1, to=2-3]
	\arrow["{T(p_M)}", from=1-3, to=2-3]
	\arrow["{0_M}"', from=2-1, to=2-3]
\end{tikzcd}\]
is a pullback diagram. In \cite{Cockett2014DifferentialST} Cockett and Cruttwell say that ``while being perhaps the least intuitive aspect of tangent structure, [the universality of the vertical lift] is
also perhaps the most important".

The lack of intuition for the universality of the vertical lift makes it hard to construct additional examples of tangent categories or check if a certain endofunctor is part of a tangent structure.
 
However, in the category of smooth manifolds, pullbacks along submersions fulfill that adding the dimensions along both diagonals gives the same number. It follows from a short argument that the universality of the vertical lift restricts the dimension of the tangent bundle $\dim(T(M))$ to be either $\dim(M)$ or $2 \dim(M)$.

Analogously, in Definition \ref{def:monoid-valued-dimension} we define categorical dimensions on other categories and interpret the universality of the vertical lift in terms of them.

The point is not to generalize existing notions of dimension in many different categories(though it generalizes the dimension of manifolds and the rank of modules), but to interpret the universality of the vertical lift and put restrictions on possible tangent structures.

The main result is Theorem \ref{thm:dimension_result_weak} which states that given a dimension $\dim$ on a tangent category $(\mathbb X,T)$, for any object $X$, the equation
$$
\dim(T^2(X)) + 2 \cdot \dim (X) = 3 \cdot \dim(T(X))
$$
holds. The second main result, Theorem \ref{thm:dimension_argument_strong}, states that under certain circumstances $\dim(T(X)) = \dim(X)$ or $\dim(T(X)) = 2 \cdot \dim(X)$ has to hold. Both of them are consequences of the universality of the vertical lift in the presence of  a dimension.

After revisiting the background on tangent categories in Section \ref{sec:tangent_cats}, we will continue by defining a generalized notion of dimension in Section \ref{sec:dimension_def}. One key aspect of this generalization is that the dimension can be an element of any commutative monoid, not just the natural numbers with addition. In subsection \ref{subsec:tangent_and_dimension}, we will prove Theorems \ref{thm:dimension_result_weak} and \ref{thm:dimension_argument_strong}, the main results of this paper. 
 
These results allow us to narrow down the search for tangent structures on a given category (with dimension). For different underlying categories $\mathbb X$ obtain results of the form ``any tangent structure on $\mathbb X$ fulfills this equation". 

We give several examples of categorical dimensions that go beyond the motivating example of smooth manifolds.
In Section \ref{sec:sets} we observe that the cardinality is a categorical dimension on $\Set^\mathrm{op}$ and use this to show that the trivial tangent structure is the only Cartesian tangent structure on $\Set^\mathrm{op}$.
In Section \ref{sec:groups}, we observe that the cardinality is a categorical dimension on $\Grp$. However, instead of being valued in the monoid given by addition, it is valued in the monoid given by multiplication, demonstrating the versatility of our use of monoids in the definition.

In Section \ref{sec:rings} we observe that, while the Krull dimension is not a categorical dimension for rings, the characteristic of rings is a categorical dimension. Interestingly, the monoid this dimension is valued in is given neither by addition nor multiplication, it is the least common multiple (lcm).

In Section \ref{sec:modules} we observe that the rank of a module over a ring is a categorical dimension. Using the example of the free forgetful adjunction between modules and algebras, we showcase the compatibility of dimensions under adjunctions.
In addition, in the case of vector spaces the rank is the classical dimension. The dimension can be used to prove that there are only two Cartesian tangent structures on the category $\Vect_F$ of vector spaces over a field $F$.

In Section \ref{sec:cw} we observe that the dimension of CW complexes is not a categorical dimension. However, the Betti numbers are a dimension on the opposite category of CW complexes $\CW^\mathrm{op}$. 
In Proposition \ref{prop:excisive_funcotrs_preserve_dimension} we show that, under certain circumstances, dimensions can be transported along excisive functors. Since excisive functors can be thought of as a generalization of homology, this hints at a deep connection between categorical dimensions and homotopy theory.

We include many non-examples which stem from concepts that may intuitively evoke the idea of dimension, but which turn out not to be categorical dimensions.  This highlights the main purpose of the concept that we are introducing in this paper: the dimension is meant to provide context for the universality of the vertical lift.


\noindent \textbf{Acknowledgements} I would like to thank Marcello Lanfranchi for suggesting to write this paper. For their ideas, comments, questions and discussions, I would like to thank Geoff Vooys, David Spivak, Kristaps Balodis, Jose Cruz and Tracey Balehowsky. I would also like to thank  Remy van Dobben de Bruyn for his comment on mathoverflow, pointing out a mistake in an earlier draft.
Special thanks goes to my PhD supervisor Kristine Bauer for her extensive writing support. While developing this project, I was funded by the Alberta Innovates Graduate Student Scholarship.

\section{Tangent categories}\label{sec:tangent_cats}
This section aims to recall notions from \cite{Cockett2014DifferentialST}, namely tangent categories, lax/strong tangent functors between tangent categories and (linear) tangent transformations between tangent functors. Tangent categories were originally defined by Rosick\'y in \cite{rosicky} in a slightly different way; the formulation we use is the one from \cite{Cockett2014DifferentialST}.

\begin{definition}\label{def:additive_bundle}\cite[Definition 2.1]{Cockett2014DifferentialST}
Let $\mathbb X$ be a category and $A\in \mathbb X_0$ an object. An \textbf{additive bundle} $(X,A,+, 0,p)$ \textbf{over $A$} consists of the following data:
\begin{itemize}
    \item an object $X$ and a morphism $p: X\rightarrow A$ such that pullback powers $X_n$ of $p$ exist, and 
    \item morphisms $+ : X_2 \rightarrow X$ and $0: A \rightarrow X$. 
\end{itemize}

In addition, the morphisms must satisfy the following conditions:
\begin{itemize}
\item $p \circ += p \circ \pi_0 = p \circ \pi_1$ and $p \circ 0= 1_A$,
\item the morphism $+$ must be associative, commutative and unital.  
\end{itemize}
That is, the following associativity, commutativity and unitality diagrams 
\begin{center}
\begin{tikzpicture}
\path (0,1.5) node(a) {$X_3$}
(2.5,1.5) node (b) {$X_2$}
(0,0) node (c) {$X_2$}
(2.5,0) node (d) {$X$};
\draw [->] (a) -- node[above] {$1 \times_A +$} (b);
\draw [->] (a) -- node[left] {$+ \times_A 1$} (c);
\draw [->] (c) -- node[above] {$+$} (d);
\draw [->] (b) -- node[left] {$+$} (d);
\end{tikzpicture}
\begin{tikzpicture}
\path (0,1.5) node(a) {$X_2$}
(0,0) node (c) {$X_2$}
(2.5,0) node (d) {$X$};
\draw [->] (a) -- node[above] {$+$} (d);
\draw [->] (a) -- node[left] {$\langle \pi_1 , \pi_0 \rangle $} (c);
\draw [->] (c) -- node[above] {$+$} (d);
\end{tikzpicture}
\begin{tikzpicture}
\path (0,1.5) node(a) {$X$}
(0,0) node (c) {$X_2$}
(2.5,0) node (d) {$X$};
\draw [->] (a) -- node[above] {$1_X$} (d);
\draw [->] (a) -- node[left] {$\langle 0 \circ p , 1_X \rangle $} (c);
\draw [->] (c) -- node[above] {$+$} (d);
\end{tikzpicture}
\end{center}
must commute.
\end{definition}

A morphism of additive bundles consists of morphisms that commute with these structures, as explained in the following definition.
\begin{definition}\label{def:additive_bundle_morphism}\cite[Definition 2.2]{Cockett2014DifferentialST}
For two additive bundles $(X,A, p, + , 0)$ and $(Y,B,p',+',0')$ an \textbf{additive bundle morphism} $(f,g)$ is a pair of maps $f: X \rightarrow Y$ and $g: A \rightarrow B$ such that the following diagrams commute:
$$
\begin{tikzpicture}
\path (0,1.5) node(a) {$X$}
(2.5,1.5) node (b) {$Y$}
(0,0) node (c) {$A$}
(2.5,0) node (d) {$B$};
\draw [->] (a) -- node[above] {$f$} (b);
\draw [->] (a) -- node[left] {$p$} (c);
\draw [->] (c) -- node[above] {$g$} (d);
\draw [->] (b) -- node[left] {$p'$} (d);
\end{tikzpicture}
\begin{tikzpicture}
\path (0,1.5) node(a) {$X_2$}
(2.5,1.5) node (b) {$Y_2$}
(0,0) node (c) {$X$}
(2.5,0) node (d) {$Y$};
\draw [->] (a) -- node[above] {$\langle f \circ \pi_0 , f \circ \pi_1 \rangle$} (b);
\draw [->] (a) -- node[left] {$+$} (c);
\draw [->] (c) -- node[above] {$f$} (d);
\draw [->] (b) -- node[left] {$+'$} (d);
\end{tikzpicture}
\begin{tikzpicture}
\path (0,1.5) node(a) {$A$}
(2.5,1.5) node (b) {$B$}
(0,0) node (c) {$X$}
(2.5,0) node (d) {$Y$};
\draw [->] (a) -- node[above] {$g$} (b);
\draw [->] (a) -- node[left] {$0$} (c);
\draw [->] (c) -- node[above] {$f$} (d);
\draw [->] (b) -- node[left] {$0'$} (d);
\end{tikzpicture}
$$
\end{definition}
Additive bundles are a categorification of the basic structure underlying vector spaces.  This will be helpful in defining tangent categories, since each tangent space in the classical case is a vector space.  In the following definition, beware of the notation $T_2$ (which refers to a pullback) and $T^2$ (which refers to two applications of the functor $T$).  For any category $\mathbb X$, let $1:\mathbb X \to \mathbb X$ denote the identity functor.  

\begin{definition}\cite[Definition 2.3]{Cockett2014DifferentialST}
\label{def:tangent_cat}
A \textbf{tangent category} $(\mathbb X , T , p , 0,+, \ell , c)$ consists of a category $\mathbb X$ and the following structures:
\begin{itemize}
\item a functor $T: \mathbb X \rightarrow \mathbb X$, called the tangent functor with a natural transformation $p: T \rightarrow 1$ such that pullback powers of $p_M:TM\to M$ exist for all objects $M \in \mathbb X$ and $T$ preserves these pullback powers;
\item natural transformations $+: T_2 \rightarrow T$ and $0: 1 \rightarrow T$ making each $(TM, M, p_M, +_M, 0_M)$ 
an additive bundle;
\item a natural transformation $\ell: T \rightarrow T^2$, called the vertical lift, such that $(\ell_M, 0_M)$ is an additive bundle morphism from $(TM, M, p_M, +_M, 0_M)$ to $(T^2M, TM, Tp_M, T(+_M), T(0_M))$;
\item a natural transformation $c: T^2 \rightarrow T^2$ with $c^2 = 1$ and $lc = l$, such that $(c_M, 1)$ is an additive bundlem morphism from $(T^2M, TM, Tp_M, T(+_M), T(0_M))$ to $(T^2M, TM, p_{TM}, +_{TM}, 0_{TM})$.
\end{itemize}
The pullbacks $T^n T_k$ are denoted as \textbf{foundational pullbacks}.

In addition, $\ell$ and $c$ are required to be compatible with themselves and with each other, meaning that the diagrams
\[\begin{tikzcd}[column sep=2.25em]
	TM && {T^2M} & {T^3M} & {T^3M} & {T^3M} & {T^2M} & {T^3M} & {T^3M} \\
	{T^2M} && {T^3M} & {T^3M} & {T^3M} & {T^3M} & {T^2M} && {T^3M}
	\arrow["{\ell_{TM}}"', from=2-1, to=2-3]
	\arrow["\ell", from=1-1, to=1-3]
	\arrow["\ell"', from=1-1, to=2-1]
	\arrow["{T(\ell)}", from=1-3, to=2-3]
	\arrow["{T(c)}", from=1-4, to=1-5]
	\arrow["{c_{TM}}", from=1-5, to=1-6]
	\arrow["{c_{TM}}"', from=1-4, to=2-4]
	\arrow["{T(c)}"', from=2-4, to=2-5]
	\arrow["{c_{TM}}"', from=2-5, to=2-6]
	\arrow["{T(c)}", from=1-6, to=2-6]
	\arrow["{\ell_{TM}}", from=1-7, to=1-8]
	\arrow["{T(c)}", from=1-8, to=1-9]
	\arrow["c"', from=1-7, to=2-7]
	\arrow["{T(\ell)}"', from=2-7, to=2-9]
	\arrow["{c_{TM}}", from=1-9, to=2-9]
\end{tikzcd}\]
all commute.  Finally, we require that  
\begin{equation}\label{diag:universality_tangent_structure}
\begin{tikzcd}
	{T_2M} && {T^2M} \\
	M && TM
	\arrow["\nu", from=1-1, to=1-3]
	\arrow["{p_M \circ \pi_0}"', from=1-1, to=2-1]
	\arrow["\lrcorner"{anchor=center, pos=0.125}, draw=none, from=1-1, to=2-3]
	\arrow["{T(p_M)}", from=1-3, to=2-3]
	\arrow["{0_M}"', from=2-1, to=2-3]
\end{tikzcd}
\end{equation}
is a pullback preserved by powers of $T$, where $\nu$ is a shorthand notation for $\nu = T(+) \circ \langle \ell_M \circ \pi_0 , 0_{TM} \circ \pi_1 \rangle $. This pullback is denoted as the \textbf{universality of the vertical lift}. The foundational pullbacks together with powers of $T$ applied to the universality of the vertical lift will be called the \textbf{tangent pullbacks}.
\end{definition}
While the last condition is formulated differently in Definition 2.3 of \cite{Cockett2014DifferentialST}, it is shown to be equivalent to the formulation here in Lemma 2.12 of \cite{Cockett2014DifferentialST}.

\begin{remark}
The last condition is called the universality of the vertical lift and it encodes the intuition that $T^2M$ relates to $T_2M$ in the same way that $TM$ relates to $M$. For example, in the category $\SmMan$ of manifolds and smooth maps, each object has a dimenion.  

The vertical lift ensures that $\dim(T^2M)-\dim(T_2M) = \dim (TM) - \dim(M)$. In particular if the dimension of the tangent bundle is  a multiple by $n$ of the  dimension of the manifold, this becomes $n^2-(2n-1) = n-1$, which only is the case when $n=1$ or $n=2$. 
Thus the universality of the vertical lift in \SmMan{} enforces the idea that the dimension of the tangent bundle is either the same or twice  the dimension of the original manifold. We will explore the relationship between the vertical lift and the dimension of the tangent bundle in general tangent categories in the later sections of this paper.
\end{remark}

The structures involved in Definition \ref{def:tangent_cat} of tangent categories can be summarized as a diagram:
\[\begin{tikzcd}
	& {T_2(M)} \\
	M & {T(M)} & {T^2(M)}
	\arrow["{+}", from=1-2, to=2-2]
	\arrow["0", shift left, from=2-1, to=2-2]
	\arrow["p", shift left, from=2-2, to=2-1]
	\arrow["\ell", shift left, from=2-2, to=2-3]
	\arrow["{T(p)}", shift left, from=2-3, to=2-2]
	\arrow["c", from=2-3, to=2-3, loop, in=55, out=125, distance=10mm]
\end{tikzcd}\]
for every object $M$.

The main motivating example for the definition are smooth manifolds. This tangent structure is explained in  \cite[Section 2.2]{Cockett2014DifferentialST}. Another example for a tangent structure is described below.

A Cartesian tangent category is a tangent category in which finite products exist and the product projections are compatible with the tangent structure, see \cite[Section 2.4]{Cockett2014DifferentialST}.

\begin{example}\label{ex:trivial_tangent_structure}\cite[Section 2.2]{Cockett2014DifferentialST}
Any category $\mathbb X$ with the identity functor and identity transformations for all of the natural transformations required by Defintion \ref{def:tangent_cat} give an example of a tangent category, $(\mathbb X, 1_\mathbb X, 1_{1_{\mathbb X}} , 1_{1_{\mathbb X}}, 1_{1_{\mathbb X}}, 1_{1_{\mathbb X}},1_{1_{\mathbb X}})$.
\end{example}

From this example we see that, in general, a given category $\mathbb X$ can have multiple distinct tangent structures.
For example the category of smooth manifolds has the classical tangent bundle structure from differential geometry and the trivial tangent structure.
In fact, any category with a nontrivial tangent structure $T\neq 1_\mathbb X$ has additionally the trivial tangent structure, demonstrating that it has at least two distinct tangent structures.

A central object of study in tangent categories are differential bundles, the categorical generalization of the notion of vector bundles from differential geometry. 
We now recall the definition of a differential bundles from \cite{cockett2016diffbundles}:
\begin{definition}\label{def:differential bundle}\cite[Definition 2.3]{cockett2016diffbundles}
Given a tangent category $(\mathbb X,T)$, a \textbf{differential bundle} \\ $(E,M,q,\zeta, \sigma, \lambda)$ consists of 
\begin{enumerate}
    \item an object $E$ (the total space),
    \item an object $M$ (the base space),
    \item a morphism $q:E \to M$ (the projection), admitting finite pullback powers $\{E_n\}_{n \in \mathbb N}$ over itself that are preserved by powers of the tangent functor $T^k$,
    \item a morphism $\zeta: M \to E$ (the zero section) and a morphism $\sigma: E_2 \to E$ (fiberwise addition) such that $(q,\zeta,\sigma)$ is an additive bundle, and
    \item a morphism $\lambda: E \to T(E)$ (the lift) such that $$(\lambda, 0): (E,M,q, \sigma,\zeta) \to (T(E),T(M),T(q)), T(\sigma), T(\zeta))$$ and $$(\lambda, \zeta): (E,M,q, \sigma,\zeta) \to (T(E),E,p, + , 0)$$ are additive bundle morphisms.
\end{enumerate}
In addition the diagram
\begin{equation}
\begin{tikzcd}
	{E_2} &&& TE \\
	M &&& TM
	\arrow["{q\circ \pi_0}"', from=1-1, to=2-1]
	\arrow["0", from=2-1, to=2-4]
	\arrow["{T(q)}", from=1-4, to=2-4]
	\arrow["{T(\sigma) \circ \langle \lambda \circ \pi_0 , 0 \circ \pi_1\rangle }", from=1-1, to=1-4]
	\arrow["\lrcorner"{anchor=center, pos=0.125}, draw=none, from=1-1, to=2-4]
\end{tikzcd}\label{diagram:universality_diffbun}
\end{equation}
is a pullback and the morphisms $\ell_E \circ \lambda , T(\lambda) \circ \lambda : E \to T(T(E))$ are equal.
\end{definition}

\section{Categorical dimensions of tangent bundles}\label{sec:dimension_def}

 In this section, we will present a definition of dimension in categories which generalizes the definition of dimension for smooth manifolds. The dimension of manifolds has many desirable properties, some of which we will list here. 
 \begin{remark}\label{rem:wishlist_dimension}
    Let $X, X'$ and $X''$ be smooth manifolds. The dimension of manifolds has certain properties which we would consider desirable for generalizing dimensions.
     \begin{enumerate}
     \item[\mylabel{item:dim_is_natural_number}{(a)}]\textbf{(valued in $\mathbb N$)} The dimension assigns manifolds to natural numbers $\Dim(X) \in \mathbb N = \{0,1,2...\}$.
     \item[\mylabel{item:pullbacks}{(b)}] \textbf{(additive w.r.t. pullbacks)} The dimension is additive under certain pullbacks: Given a cospan $X\rightarrow X'\leftarrow X''$ of transversal maps, $\Dim(X \times_{X'} X'') + \Dim(X') = \Dim(X) + \Dim(X'')$
     \item[\mylabel{item:cartesian_products_to_sums}{(b')}] \textbf{(additive w.r.t. products)} The dimension is additive under Cartesian products: $\Dim(X \times X') = \Dim(X) + \Dim(X')$
     \item[\mylabel{item:isomorphism}{(c)}] \textbf{(isomorphism invariant)} The dimension is invariant under isomorphism: Given $X \cong X'$, then $\Dim(X) = \Dim(X')$.
\end{enumerate}    
The topology and tangent structure of smooth manifolds also has implications for the dimension.
\begin{enumerate}
     \item[\mylabel{item:embeddings_inequalities}{(e)}] \textbf{(embeddings to inequalities)} The dimension sends embeddings to inequalities: Given an embedding $X \hookrightarrow X'$, $\Dim(X) \leq \Dim(X')$
     \item[\mylabel{item:locality}{(e')}] \textbf{(the same as locally)} The dimension is local: Given an open subset $U \hookrightarrow X$, $\Dim(U) = \Dim(X)$.
     \item[\mylabel{item:twice_dim}{(f)}] \textbf{(tangent bundle doubles dim)} Applying the tangent bundle endofunctor doubles the dimension: $\Dim(TX) = 2 \cdot \Dim(X)$.
     \item[\mylabel{item:differential_any_fiber}{(g)}] \textbf{(bundle fiber arbitrary dim)} A vector bundle can have a fiber of any dimension: For any $n \in \mathbb N$, there is a vector bundle $E \to X$ such that $\dim(E) = \dim (X) + n$.
 \end{enumerate}     
 \end{remark}

We will take property \ref{item:pullbacks} (additive w.r.t. pullbacks) as the definition of categorical dimensions. The reason we choose this condition is that we want to apply the properties of a dimension to the pullback diagram
\[\begin{tikzcd}
	{T_2X} && {T^2X} \\
	X && TX
	\arrow[from=1-1, to=1-3]
	\arrow[from=1-1, to=2-1]
	\arrow["\lrcorner"{anchor=center, pos=0.125}, draw=none, from=1-1, to=2-3]
	\arrow["{T(p_X)}", from=1-3, to=2-3]
	\arrow["{0_X}"', from=2-1, to=2-3]
\end{tikzcd}\]
encoding the universality of the vertical lift. We will replace the notion of transversality (which inherently depends on the tangent structure) with a more general notion, in order to define a dimension on a category without knowing the tangent structure. While we will not assume desired Property \ref{item:cartesian_products_to_sums} (additive w.r.t. products), however if the dimension of the terminal object is the neutral element $\dim(\ast)=0$, it follows from Property \ref{item:pullbacks} (additive w.r.t. pullbacks).

We generalize Property \ref{item:dim_is_natural_number} (valued in $\mathbb N$) by asking the target of the dimension to be a commutative monoid.  This is partially motivated by existing generalizations of dimension (such as infinite dimension and fractional dimension). We will provide examples in Sections \ref{sec:sets}--\ref{sec:cw} of a notion of dimension which satisfies the definition and is not an integer.

Since none of our results rely on it, we also do not require Property \ref{item:embeddings_inequalities} (embeddings to inequalities).
We also choose to not require property \ref{item:locality}(the same as locally) since we want to avoid using any notion of locality. 

Our definition of dimension will not be unique. For example, the constant assignment $\dim(X) = 0, \forall X$ will be a dimension and so is any other constant assignment $\dim(X) = 42, \forall X$.

As a convention, we denote the classical notion of dimension of manifolds and vector spaces as $\Dim$ with a capital D, whereas we denote the generalized dimension as $\dim$ with lowercase d.
\subsection{Monoid valued dimensions}\label{subsec:monoid_dimension_def}
In this section we will define our categorical notion of dimension. The key will be how the dimension behaves under the pullbacks
\[\begin{tikzcd}
	{T_2M} & TM & {T_2M} & {T^2M} \\
	TM & M & M & TM
	\arrow[from=1-1, to=1-2]
	\arrow[from=1-1, to=2-1]
	\arrow["\lrcorner"{anchor=center, pos=0.125}, draw=none, from=1-1, to=2-2]
	\arrow["p", from=1-2, to=2-2]
	\arrow[from=1-3, to=1-4]
	\arrow[from=1-3, to=2-3]
	\arrow["\lrcorner"{anchor=center, pos=0.125}, draw=none, from=1-3, to=2-4]
	\arrow["{T(p)}", from=1-4, to=2-4]
	\arrow["p"', from=2-1, to=2-2]
	\arrow["0"', from=2-3, to=2-4]
\end{tikzcd}\]
which are the definition of $T_2M$ and the universality of the vertical lift. Thus we specify the behaviour for pullbacks of this form.

\begin{definition}\label{def:monoid-valued-dimension}
    Let $M$ be a commutative monoid and $\mathbb X$ be a category. Then an \textbf{$M$-valued dimension} on $\mathbb X$ is a function $\dim : \pi_0(\mathbb X) \to M$ from the set of isomorphism classes of objects of $\mathbb X$ to the monoid $M$, satisfying the following property:\\
        For any pullback 
    \[\begin{tikzcd}
    	{A \times_B C} & A \\
    	C & B
    	\arrow[from=1-1, to=1-2]
    	\arrow[from=1-1, to=2-1]
    	\arrow["{r}", from=1-2, to=2-2]
    	\arrow["{f}"', from=2-1, to=2-2]
    \end{tikzcd}\]
        of $f$ along a retraction $r$,  
    \begin{itemize}
        \item if $f$ is a retraction, or
        \item if $f$ is a section,
    \end{itemize} the equation
        $$
        \dim(A \times_B C) + \dim(B) = \dim(A) + \dim(C)
        $$
        holds.
\end{definition}
While we chose to loosen desired Property \ref{item:dim_is_natural_number} (valued in $\mathbb N$) of Remark \ref{rem:wishlist_dimension}, Property \ref{item:isomorphism} (isomorphism invariant) and a variation of Property \ref{item:pullbacks} (additive w.r.t. pullbacks) holds by definition.

In a category with products, the morphism to the terminal object, $\ast$,  is a retraction and thus any pullback over the terminal object is one of the pullbacks in Definition \ref{def:monoid-valued-dimension}. If the dimension of the terminal object satisfies $\dim(\ast) = 0$, we obtain $\dim(X \times Y) = \dim(X) + \dim(Y)$.  That is, if $\dim(\ast)=0$ property \ref{item:cartesian_products_to_sums} (additive w.r.t. products) of Remark \ref{rem:wishlist_dimension} holds.

We actively rejected to generalize property \ref{item:locality}(the same as locally) of Remark \ref{rem:wishlist_dimension} in order to disregard notions of locality. Since we, in general, do not have an order on the monoid $M$, property \ref{item:embeddings_inequalities} (embeddings to inequalities) does not hold in general. However we will see that it holds for all the examples we have listed since we have ordered monoids there.

Proving properties similar to \ref{item:twice_dim}(tangent bundle doubles dim) and \ref{item:differential_any_fiber}(bundle fiber arbitrary dim) of Remark \ref{rem:wishlist_dimension} will be the topic of Section \ref{subsec:tangent_and_dimension}.

In the following proposition we see that dimensions interact well with functors and in particular adjunctions.

\begin{proposition}
    Let $\mathbb X$ and $\mathbb Y$ be categories, let $\mathbb X$ have an $R$-valued dimension $\dim_\mathbb X: \pi_0(\mathbb X) \to R$. Let $F: \mathbb Y \to \mathbb X$ be a functor that preserves finite limits. Then the assignment that sends $Y \in \pi_0(\mathbb Y)$ to
    $$
    Y \mapsto \dim_\mathbb X(F(Y))
    $$
    is an $R$-valued dimension on $\mathbb Y$.
\end{proposition}
\begin{proof}
    Functors preserve composition and identity, thus $F$ preserves sections and retractions.
    Since it is right-adjoint $F$ preserve limits, in particular pullbacks. Thus for any pullback in $\mathbb Y$ of a section or retraction along a retraction, the dimension formula still holds.
\end{proof}
Two important classes of examples of functors that preserve finite limits are full subcategories and right-adjoint functors.
\begin{corollary}\label{cor:adjunctions_full_subcats_and_dimensions}
    Let $\mathbb X$ and $\mathbb Y$ be categories, let $\mathbb X$ have an $R$-valued dimension $\dim_\mathbb X: \pi_0(\mathbb X) \to R$. 
    \begin{enumerate}[label = (\alph*)]
        \item Let $F: \mathbb Y \to \mathbb X$ be a functor which is a right adjoint. Then the assignment that sends $Y \in \pi_0(\mathbb Y)$ to
    $$
    Y \mapsto \dim_\mathbb X(F(Y))
    $$
    is an $R$-valued dimension on $\mathbb Y$.
    \item     Let $\mathbb Y$ be a full subcategory of $\mathbb X$. Let $\dim:  \pi_0(\mathbb X) \to R$ be a $R$-valued dimension on $\mathbb X$. Then $\dim|_{\Obj(\mathbb Y)}$ is an $R$-valued dimension on $\mathbb Y$.
    \end{enumerate}
\end{corollary}
Since many adjunctions are well known from the literature, this inconspicuous corollary is a useful way to construct new dimensions from existing ones and will be frequently used throughout the paper.
One major application of Corollary \ref{cor:adjunctions_full_subcats_and_dimensions} will be in Section \ref{subsec:Alg_R}. There we will show that the forgetful functor from algebras over a ring $R$ to modules over $R$ has a left-adjoint and thus a dimension on the category of modules will induce a dimension on the category of Algebras.

Recall that in a tangent category $(\mathbb X, T)$,  a pullback is called a $T$-pullback if it is preserved by all powers of $T$. $T$-pullbacks were first developed by McAdam in \cite[Definition 1.1.3]{McAdam_arxiv} and later analyzed by Marcello Lanfranchi and Cruttwell in \cite{lanfranchi2025tangentdisplaymaps}.
We will use the definition of a $T$-pullback to define a slightly weaker version of dimension and a slightly stronger version of dimension. These definitions will be the setup of Theorem \ref{thm:dimension_result_weak} and Theorem \ref{thm:dimension_argument_strong}.
\begin{definition}\label{def:tangent_dimension}
    Let $M$ be a commutative monoid and $(\mathbb X , T)$ be a tangent category. Then an \textbf{$M$-valued tangent dimension} on $(\mathbb X, T)$ is a function $\dim : \pi_0(\mathbb X) \to M$ satisfying the following property:\\
    For any $T$-pullback 
    \[\begin{tikzcd}
    	{A \times_B C} & A \\
    	C & B
    	\arrow[from=1-1, to=1-2]
    	\arrow[from=1-1, to=2-1]
    	\arrow["{r}", from=1-2, to=2-2]
    	\arrow["{f}"', from=2-1, to=2-2]
    \end{tikzcd}\]
        along a retraction $r$,
        \begin{itemize}
        \item if $f$ is a retraction, or
        \item if $f$ is a section,
        \end{itemize} the equation
        $$
        \dim(A \times_B C) + \dim(B) = \dim(A) + \dim(C)
        $$
        holds.
    An $M$-valued tangent dimension is \textbf{strong} if $M$ is a commutative rig and there is an element $a \in \mathbb M$ such that for every object $X \in \mathrm{Ob}(\mathbb X)$ the equation
    $$
    \dim (T(X)) = a \cdot \dim(X)
    $$
    holds.
\end{definition}
Since every T-pullback is a pullback, it is immediate that every $M$-valued dimension is an $M$-valued tangent dimension (regardless of the tangent structure) but not vice versa. 

Again, $M$-valued tangent dimensions satisfy Properties \ref{item:pullbacks} (additive w.r.t. pullbacks), \ref{item:isomorphism} (isomorphism invariant) of Remark \ref{rem:wishlist_dimension}. We will explore in the next sections what the properties corresponding to \ref{item:twice_dim} (tangent bundle doubles dim) and \ref{item:differential_any_fiber} (bundle fiber arbitrary dim) are.

\begin{example}
    Given any tangent category $(\mathbb X, T, p, 0, + ,\ell, c)$ and any commutative monoid $M$, let $m \in M$ be any element of $M$. Then the constant assignment $\dim(X)= m$ is a $M$-valued dimension and a strong $M$-valued tangent dimension. All the equations become trivial in this situation. We will call dimensions of this form \textbf{trivial dimensions} on $\mathbb X$.
\end{example}

We explore the Properties of $M$-valued tangent dimensions in Theorems 3.7 and 3.8.

\subsection{The tangent structure in the presence of a dimension}\label{subsec:tangent_and_dimension} 

If we have a tangent dimension on a category, it interacts with the tangent structure since the tangent structure requires certain diagrams to be pullbacks. We see in Theorem \ref{thm:dimension_argument_strong} that in certain situations the dimension of the tangent bundle behaves in a way similar to desirable Property \ref{item:twice_dim}(tangent bundle doubles dim) of Remark \ref{rem:wishlist_dimension}.

\begin{theorem}\label{thm:dimension_result_weak}
    Let $M$ be a commutative monoid.
    Let $(\mathbb X , T)$ be a tangent category. Let
    $\dim: \mathrm{Ob}(\mathbb X) \to M$ be a $M$-valued tangent dimension. Then for every object $X \in \mathrm{Ob}(\mathbb X)$
        $$
        \dim(T^2(X)) + 2 \cdot \dim (X) = 3 \cdot \dim(T(X)).
        $$
    \end{theorem}
    \begin{proof}
    We consider the following two pullback diagrams:
\[\begin{tikzcd}
	{T_2X} & TX && {T_2X} && {T^2X} \\
	TX & X && X && TX
	\arrow[from=1-1, to=1-2]
	\arrow[from=1-1, to=2-1]
	\arrow["\lrcorner"{anchor=center, pos=0.125}, draw=none, from=1-1, to=2-2]
	\arrow["p", from=1-2, to=2-2]
	\arrow["{\nu}", from=1-4, to=1-6]
	\arrow["{\pi_0 p}"', from=1-4, to=2-4]
	\arrow["\lrcorner"{anchor=center, pos=0.125}, draw=none, from=1-4, to=2-6]
	\arrow["{T(p)}", from=1-6, to=2-6]
	\arrow["p"', from=2-1, to=2-2]
	\arrow["0"', from=2-4, to=2-6]
\end{tikzcd}\]
where $\nu$ is a shorthand notation for $\nu = T(+) \circ \langle \ell_M \circ \pi_0 , 0_{TM} \circ \pi_1 \rangle$.
The first diagram is a pullback since this is the definition of $T_2M$. The second diagram is a pullback, because this is an equivalent formulation for the universality of the vertical lift.

Since $0:X \to T(X)$ and $p: T(X) \to X$ form a section retraction pair, the pullbacks above fulfill the requirements in Definition \ref{def:monoid-valued-dimension} we obtain that
$$
\dim(T_2(X)) + \dim(X) = 2 \dim(T(X)) \text{ and } \dim(T_2(X)) + \dim(T(X)) = \dim(T^2(X)) + \dim (X).
$$
Therefore we obtain 
$$
2 \cdot \dim(T(X)) + \dim(T(X)) = \dim(T_2(X)) + \dim(X) + \dim(T(X)) = \dim(T^2(X)) + \dim(X) + \dim(X).
$$
\end{proof}

\begin{theorem}\label{thm:dimension_argument_strong}
Let $R$ be an integral domain (unital commutative ring without zero divisors).
    Let $(\mathbb X , T)$ be a tangent category. Let
$\dim: \mathrm{Ob}(\mathbb X) \to R$ be a strong $R$-valued tangent dimension on $(\mathbb X,T)$.
Then for any object $X$, $\dim(T(X)) = \dim(X)$ or $\dim(T(X)) = 2 \cdot \dim(X)$ holds.
\end{theorem}
If there exists an object with $\dim(X) \neq 0$ this means $a=1$ or $a=2$.
\begin{proof}
    We consider the following two pullback diagrams:
\[\begin{tikzcd}
	{T_2X} & TX && {T_2X} && {T^2X} \\
	TX & X && X && TX
	\arrow[from=1-1, to=1-2]
	\arrow[from=1-1, to=2-1]
	\arrow["\lrcorner"{anchor=center, pos=0.125}, draw=none, from=1-1, to=2-2]
	\arrow["p", from=1-2, to=2-2]
	\arrow["{\nu}", from=1-4, to=1-6]
	\arrow["{\pi_0 p}"', from=1-4, to=2-4]
	\arrow["\lrcorner"{anchor=center, pos=0.125}, draw=none, from=1-4, to=2-6]
	\arrow["{T(p)}", from=1-6, to=2-6]
	\arrow["p"', from=2-1, to=2-2]
	\arrow["0"', from=2-4, to=2-6]
\end{tikzcd}\]
where $\nu$ is a shorthand notation for $\nu = T(+) \circ \langle \ell_M \circ \pi_0 , 0_{TX} \circ \pi_1 \rangle$.

The first diagram is a pullback since this is the definition of $T_2M$. The second diagram is a pullback, because this is an equivalent formulation for the universality of the vertical lift. Therefore we obtain that
$$
\dim(T(X)) + \dim(T(X))- \dim(X) = \dim(T_2X) = \dim(T^2X) + \dim(X) - \dim(TX).
$$
This can be reformulated into
$$
\dim(T^2 X ) - 3 \dim(TX) + 2 \dim(X) = 0.
$$
Using that  $\dim(TX) = a \cdot \dim(X)$, we obtain the equation
\begin{align*}
    0 &= a^2 \cdot \dim(X) -3a \cdot \dim(X) +2 \dim(X)
    \\
    0 &= (a-1) \cdot (a-2) \cdot \dim(X)
\end{align*}
This implies $a=1$, $a=2$ or $\dim(X)=0$. If $a=1$, then $\dim(TX) = \dim(X)$. If $a=2$, then $\dim(TX) = 2 \cdot \dim(X)$. If $\dim(X) = 0$, both hold.
\end{proof}

Theorem \ref{thm:dimension_argument_strong} is the analogue to Property (g)

\begin{example}\label{ex:manifolds}
    In the tangent category of smooth manifolds, the classical dimension $\dim(M) = \Dim(M)$  is an example for a $\mathbb Z$-valued strong tangent dimension. If the retraction $r: A \to B$ is a submersion, it is transversal to any smooth map $C \to B$. Thus, by \cite[Theorem 15.2]{brendon_topo_and_geo}, the pullback fulfills $\dim (A \times_B C) = \dim(A) + \dim(C)- \dim(B) $. Therefore Theorem \ref{thm:dimension_argument_strong} shows that for any tangent structure on SmMan with $\dim(TM) = a \cdot \dim(M)$, we obtain that $a=1$ or $a=2$. We prove the general statement that the classical dimension $\Dim$ fulfills Definition \ref{def:monoid-valued-dimension} in Proposition \ref{prop:dimension_is_dimension}, where we do not assume that $r$ is a submersion.
\end{example}
One important application of Theorem \ref{thm:dimension_result_weak} and Theorem \ref{thm:dimension_argument_strong} is that they provide a method to prove that some functor can not be part of a tangent structure. If a category has a good notion of dimension, but the proposed tangent functor fails to satisfy the conditions of Theorem \ref{thm:dimension_result_weak} and Theorem \ref{thm:dimension_argument_strong}, then the functor can not be a tangent functor. The following example, discovered by G. Vooys \cite{geoff_peripatetic}, showcases this application.
\begin{example}
    Let $\operatorname{SmMan}$ be the category of smooth real manifolds and smooth maps. Given the functor $\tilde T : \operatorname{SmMan} \to \operatorname{SmMan}$ that sends an object $M$ to the object $M \times \mathbb R$ and a morphism $f$ to $f \times 1_\mathbb R$. Then there is no tangent structure $(\operatorname{SmMan}, \hat T,p,+,0,\ell)$ for which the tangent functor is $\tilde T$. The reason is that according to Theorem \ref{thm:dimension_result_weak} the equation 
    $$
    \dim (M \times I \times I ) + 2 \cdot \dim(M) = 3 \cdot \dim (M \times I)
    $$
    would need to hold. However, this would mean that
    $$
    3 \cdot \dim(M) + 2 = 3 \cdot \dim (M) +3
    $$
    and thereby $3=2$ proving by contradiction that there is no such tangent structure.
\end{example}
In fact, for this endofunctor $\tilde T$, it is possible to define natural transformations $p,0,+,\ell, c$ that fulfill all the properties of a tangent structure, except the universality of the vertical lift.  

As we promised in Example \ref{ex:manifolds}, below we prove that the dimension of manifolds fulfills the pullback formula of Definition \ref{def:monoid-valued-dimension}. We begin with a lemma turning a general retraction into a submersion.

\begin{lemma} \label{lem:submersion_subset}
    Let $r: A \to B$ be a retraction with section $s: B \to A$. Then there is an open subset $U \subseteq A$ such that $\mathrm{Im}(s) \subset U$ and $r|_U$ is submersion.
\end{lemma}
The fact that $\mathrm{Im}(s) \subset U$ simultaneously shows that $U$ is non-empty and also that $s$ is a section of $r|_U$.
\begin{proof}
    Let $U \subset A$ be the set of regular points in $A$. 
    Due to \cite[Proposition 4.1]{Leee_intro_smooth_manifolds}, $U$ is open. By its definition, all points in $U$ are regular, therefore $r|_U$ is a submersion. 
    
    It only remains to show that $\mathrm{Im}(s) \subset U$.  Let $a = s(b)$ be in $\mathrm{Im}(s) \subset A$. In order to show that $a$ is a regular point, let $v \in T_b(B)$ be an arbitrary tangent vector. Then $T_bs (v) \in $ is a tangent vector at $a$ fulfilling that 
    $$
    T_ar \circ T_bs(v) = T_a(r \circ s)(v) = v
    $$
    Since $v$ is arbitrary, $a$ is a regular point. Since $a \in \mathrm{Im}(s)$ was arbitrary, this shows that $\mathrm{Im}(s) \subset U$, concluding the proof.
\end{proof}

\begin{proposition}\label{prop:dimension_is_dimension}
    Let $r: A \to B $ be a retraction of smooth manifolds. Suppose the pullback
    \[\begin{tikzcd}
    	{A \times_B C} & A \\
    	C & B
    	\arrow["{\pi_0}", from=1-1, to=1-2]
    	\arrow["{\pi_1}"', from=1-1, to=2-1]
    	\arrow["\lrcorner"{anchor=center, pos=0.125}, draw=none, from=1-1, to=2-2]
    	\arrow["r", from=1-2, to=2-2]
    	\arrow["f"', from=2-1, to=2-2]
    \end{tikzcd}\]
    exists. Then the equation
    $$
    \dim(A \times_B C) = \dim(A) + \dim(C)- \dim(B)
    $$
    holds.
\end{proposition}
\begin{proof}
    Due to Lemma \ref{lem:submersion_subset}, there is an open subset $U \subset A$ containing $\mathrm{Im}(s)$ such that $r|_U$ is a submersion. Since the embedding $U \hookrightarrow A$ is a submersion, the pullback 
    \[\begin{tikzcd}
    	{\pi_0^{-1}(U)} & U \\
    	{A \times_B C} & A
    	\arrow[from=1-1, to=1-2]
    	\arrow[from=1-1, to=2-1]
    	\arrow["\lrcorner"{anchor=center, pos=0.125}, draw=none, from=1-1, to=2-2]
    	\arrow[hook, from=1-2, to=2-2]
    	\arrow["{\pi_0}"', from=2-1, to=2-2]
    \end{tikzcd}\]
    exists and euqals the preimage $\pi_0^{-1}(U)$. As the preimage of an open set,  $\pi_0^{-1}(U)$ is an open subset of $A \times_B C$.

    Both inner squares in 
\[\begin{tikzcd}
	{\pi_0^{-1}(U)} & U \\
	{A \times_B C} & A \\
	C & B
	\arrow[from=1-1, to=1-2]
	\arrow[from=1-1, to=2-1]
	\arrow["\lrcorner"{anchor=center, pos=0.125}, draw=none, from=1-1, to=2-2]
	\arrow[hook, from=1-2, to=2-2]
	\arrow["{\pi_0}"', from=2-1, to=2-2]
	\arrow[from=2-1, to=3-1]
	\arrow["\lrcorner"{anchor=center, pos=0.125}, draw=none, from=2-1, to=3-2]
	\arrow["r", from=2-2, to=3-2]
	\arrow["f"', from=3-1, to=3-2]
\end{tikzcd}\]
    are pullbacks. Therefore the outer square is the pullback of $f$ along $r \circ i$. Since $r \circ i = r |_U$ is a submersion, the dimension formula
    $$
    \dim (\pi_0^{-1}(U)) = \dim(U) + \dim(C) - \dim(B)
    $$
    holds. Since the dimension of an open subset equals the dimension of the surrounding manifold, $\dim (\pi_0^{-1}(U)) = \dim(A \times_B C)$ and $\dim(U) = \dim(A)$. Therefore we obtain
    $$
    \dim(A \times_B C) = \dim(\pi_0^{-1}(U)) = \dim(U) + \dim(C) - \dim(B) = \dim(A) + \dim(C) - \dim(B)
    $$
    which concludes the proof.
\end{proof}

\subsection{Differential bundles in the presence of a dimension}\label{subsec:diff_bund_when_dimension}

For tangent structures, the dimension helps to give context to the universality of the vertical lift.  Similarly, the pullback in Diagram \ref{diagram:universality_diffbun} is the vertical lift for differential bundles.  In this section, we prove the analogues of Theorems \ref{thm:dimension_result_weak} and \ref{thm:dimension_argument_strong} for differential bundles.

In the classical case of smooth manifolds, desirable Property \ref{item:differential_any_fiber}(bundle fiber arbitrary dim) of Remark \ref{rem:wishlist_dimension} says that there is no restriction on the dimension of a differential bundle.  However, in certain non-classical circumstances there is. This will lead to a somewhat strange genaralization of desirable Property \ref{item:differential_any_fiber}(bundle fiber arbitrary dim) of Remark \ref{rem:wishlist_dimension} in Theorem \ref{thm:DiffBun_dimension_argument_strong}.

\begin{theorem}\label{thm:DiffBun_dimension_argument_weak}
    Let $M$ be a commutative monoid and let $\dim : \pi_0(X) \to M$ be a tangent dimension on a tangent category $(\mathbb X, T, p,0,+,\ell, c)$. Let $(E, M ,q, \zeta, \sigma, \lambda)$ be a differential bundle. Then the equation
    $$
    \dim(T(E)) + 2 \cdot \dim(M) = \dim(T(M)) + 2 \cdot \dim(E)
    $$
    holds.
\end{theorem}
\begin{proof}
    The diagrams
    \[\begin{tikzcd}
    	{E_2} && E && {E_2} && {T(E)} \\
    	E && M && M && {T(M)}
    	\arrow["{\pi_0}", from=1-1, to=1-3]
    	\arrow["{\pi_1}"', from=1-1, to=2-1]
    	\arrow["\lrcorner"{anchor=center, pos=0.125}, draw=none, from=1-1, to=2-3]
    	\arrow["q", from=1-3, to=2-3]
    	\arrow["{T(\sigma) \circ \langle \lambda \circ \pi_0 , 0 \circ \pi_1\rangle}", from=1-5, to=1-7]
    	\arrow["{q \circ \pi_0}"', from=1-5, to=2-5]
    	\arrow["\lrcorner"{anchor=center, pos=0.125}, draw=none, from=1-5, to=2-7]
    	\arrow["{T(q)}", from=1-7, to=2-7]
    	\arrow["q"', from=2-1, to=2-3]
    	\arrow["0"', from=2-5, to=2-7]
    \end{tikzcd}\]
    are pullbacks. Since $q$ is a retraction, $0$ is a section and $T(q)$ is a retraction, a dimension in the sense of Definition \ref{def:monoid-valued-dimension} will fulfill the equations
    \begin{align*}
        \dim(E_2) + \dim(M) &= \dim(E) + \dim(E)
        \\
        \dim(E_2) + \dim(T(M)) &= \dim(M) + \dim(T(E)).
    \end{align*}
    Thus we obtain that
    $$
    \dim(T(E)) + 2 \cdot \dim (M)= \dim(E_2) + \dim(M) + \dim(T(M)) = \dim(T(M)) + 2 \cdot \dim(E).
    $$
\end{proof}
\begin{theorem}\label{thm:DiffBun_dimension_argument_strong}
    Let $R$ be a ring and let $\dim : \pi_0(X) \to M$ be a strong tangent dimension on a tangent category $(\mathbb X, T, p,0,+,\ell, c)$ with $a \in R$ such that $\dim(T(X)) = a \cdot \dim(X)$. Then for any differential bundle $(E,M,q,\zeta, \sigma, \lambda)$ the equation
    $$
    (a-2) \dim(E) = (a-2) \dim(M)
    $$
    holds
\end{theorem}
\begin{proof}
    Due to Theorem \ref{thm:DiffBun_dimension_argument_weak} the euqation
    $$
        \dim(T(E)) + 2 \cdot \dim(M) = \dim(T(M)) + 2 \cdot \dim(E)
    $$
    has to hold. Substituting in $\dim(T(E)) = a \cdot \dim(E)$ and $\dim(T(M)) = a \cdot \dim(M)$, we obtain
    $$
        a \cdot \dim(E) + 2 \cdot \dim(M) = a \cdot \dim(M) + 2 \cdot \dim(E).
    $$
    Subtracting $2 \cdot \dim(M) + 2 \cdot \dim(E)$ on both sides leads to the desired result that
    $$
    (a-2) \dim(E) = (a-2) \dim(M).
    $$
\end{proof}
Due to Theorem \ref{thm:dimension_argument_strong}, we know that $a$ is either 1 or 2.
If $a=2$, then there is no condition on the dimension of a differential bundle. If $a=1$, then $\dim(E) = \dim (M)$, which can be interpreted as differential bundles having 0-dimensional fibre. Thus desired Property \ref{item:differential_any_fiber}(bundle fiber arbitrary dim) may only hold in tangent categories with a strong tangent dimension $a=2$. In the tangent category of smooth manifolds this is the case.

In the following sections \ref{sec:sets}-\ref{sec:cw} we list some examples of dimensions on well-known categories and the consequences stemming from the dimension. In case it is a strong dimension, the consequences for differential bundles are either no condition (if $a=2$) or that $\dim(E) = \dim(M)$ (if $a=1$). Thus we will only mention differential bundles in case we have an example of a tangent category with a tangent dimension that is not strong.

\section{Cardinality as a non-trivial tangent structure on $\Set^\mathrm{op}$ }\label{sec:sets}
All mathematical objects are constructed from sets. 
In Corollary \ref{cor:adjunctions_full_subcats_and_dimensions}, we saw that precomposition with right-adjoint functors preserves dimensions. Most algebraic structures have a free functor from sets that is left-adjoint to the fogetful functor to sets. Thus, by precomposition with the forgetful funcor, a dimension on sets would give rise to many dimensions in many different categories. 

Since sets are discrete, there is no obvious structure in sets from which we could build tangents.  Thus, we do not expect the category of sets to have a non-trivial tangent structure.

However, that does not mean that there can not be a dimension on the category of sets. In fact, as we see in Proposition \ref{prop:only_trivial_on_setop}, a dimension exists on $\Set^\mathrm{op}$ and this dimension can be used to show that there is no nontrivial tangent structure.

The dimension of a vector space (and therefore the dimension of a manifold) is defined by the cardinality of a basis. Like the dimension for vector spaces, the cardinality is the main quantity describing the size of a set.

The cardinality fulfills desirable Properties \ref{item:dim_is_natural_number} (valued in $\mathbb N$), \ref{item:isomorphism} (isomorphism invariant) of Remark \ref{rem:wishlist_dimension}, since it is valued in the natural numbers and preserved by isomorphisms. When choosing multiplication as the monoid structure, it also fulfills \ref{item:cartesian_products_to_sums} (additive w.r.t. products), since $\#(A \times B) = \#A \cdot \#B$. 
However the cardinality is not a categorical dimension in the sense of Definition \ref{def:monoid-valued-dimension} on the category $\Set$. Additionally, we will prove that it is a dimension on $\Set^\mathrm{op}$, the opposite category of sets.

As a consequence we will be able to describe potential tangent structures on the category $\Set^\mathrm{op}$. In particular we will use the dimension to prove that there is no nontrivial Cartesian tangent structure on $\Set^\mathrm{op}$. 

\begin{definition}\label{def:monoids_N_infty}
    Let $\mathbb N= \{0,1,2,...\}$ be the set of natural numbers.
    Let $\mathbb N_\infty$ be the set $\mathbb N \cup \{\infty\}$. Then we define the following three commutative monoid structures on $\mathbb N_\infty$.
    \begin{enumerate}[label=(\alph*)]
        \item The addition $+: \mathbb N_\infty \times \mathbb N_\infty \to \mathbb N_\infty$ is defined to be the usual addition $a+b$ on natural numbers and to give $\infty$ if any of the two inputs is $\infty$. Together with the neutral element $0\in \mathbb N \subset \mathbb N_\infty$ this is a commutative monoid.
        \item The multiplication $\cdot: \mathbb N_\infty \times \mathbb N_\infty \to \mathbb N_\infty$ is defined to be the usual multiplication on natural numbers and to give $\infty$ if any of the two inputs is $\infty$ (in particular $0 \cdot \infty = \infty$). Together with the neutral element $1\in \mathbb N \subset \mathbb N_\infty$ this is a commutative monoid.
        \item  The maximum $\max: \mathbb N_\infty \times \mathbb N_\infty \to \mathbb N_\infty$ is defined to output the lowest number that is greater or equal to both inputs and to give $\infty$ if any of the two inputs is $\infty$. Together with the neutral element $0\in \mathbb N \subset \mathbb N_\infty$ this is a commutative monoid.
    \end{enumerate}
\end{definition}

\begin{remark}
    The cardinality of a set is not an $\mathbb N_\infty$-valued dimension (for either of the monoid structures in Definition \ref{def:monoids_N_infty}) on the category $\Set$ of all sets and functions. The reason for this is that coproducts (disjoint unions) of sets are preserved by pullbacks and introduce some non-trivial combinatorics to the cardinality of pullbacks. A counterexample can be easily constructed by taking the disjoint union of any pullback diagram with a singleton.
\end{remark}

However, on the opposite category $\Set^\mathrm{op}$, the cardinality is a dimension. 
\begin{theorem}\label{thm:dimension_on_setop}
    For the opposite category of sets, $\Set^\mathrm{op}$ the cardinality $\dim(X) = \#X$ is an $\mathbb N$-valued dimension on $\Set^\mathrm{op}$.
\end{theorem}
\begin{proof} Since we are in the opposite category of sets, the retractions in $\Set^\mathrm{op}$ are exactly the injections of sets.
For every pushout along an injection $i$
    \[
    \begin{tikzcd}
    	{C \sqcup_A B} & B \\
    	C & A
    	\arrow[from=1-2, to=1-1]
    	\arrow[from=2-1, to=1-1]
    	\arrow[hook',"i"', from=2-2, to=1-2]
    	\arrow["f", from=2-2, to=2-1]
    \end{tikzcd}
    \]
    the set ${C \sqcup_A B}$ can be written as a disjoint union of $C$ with $B-A$. Thus
    $$
    \#(C \sqcup_A B) + \#A = \#C + \# (B-A)+ \# A = \#C + \#B
    $$
    and thereby the condition of Definition \ref{def:monoid-valued-dimension} is fulfilled.
\end{proof}

\begin{remark}\label{rem:finiteness_possible_for_sets}
    All the results about $\Set$ in this section also hold for $\FinSet$, the category of finite sets. Since $\FinSet$ is a full subcategory of $\Set$, this follows from Corollary \ref{cor:adjunctions_full_subcats_and_dimensions} and the fact that we never used infinite sets in any proof.

    In particular the cardinality is an $(\mathbb N,+)$-valued dimension on $\FinSet^\mathrm{op}$.
\end{remark}

Below, we will apply Theorems \ref{thm:dimension_result_weak} and \ref{thm:dimension_argument_strong} to the cardinality as a dimension. A particular situation are Cartesian tangent categories. Cartesian tangent categories are defined in \cite[Definition 2.8]{Cockett2014DifferentialST} as tangent categories that admit finite products and where the products are compatible with the tangent structure.

\begin{corollary}\label{cor:tangents_on_Finsetop}~
\begin{enumerate}[label = (\alph*)]
    \item Any tangent structure on $\Set^\mathrm{op}$ will fulfill that 
    $$
    \# T^2(X) + 2 \cdot \#X = 3 \cdot \#T(X).
    $$
    \item Every Cartesian tangent structure on $\FinSet^\mathrm{op}$ fulfills $\# T(X) = \#T(\{*\}) \cdot \#X$ and $\#T(\{*\})$ is either $\{*\}$ or $\{*,*\}$.
\end{enumerate}
\end{corollary}
\begin{proof}
    Part (a) is a direct application of Theorem \ref{thm:dimension_result_weak} to the cardinality.
    
    For part (b), observe that
    every Cartesian tangent structure on $\mathrm{FinSet}^\mathrm{op}$ will preserve disjoint unions (coproducts) and therefore $T(X) \cong T(\bigsqcup_{x \in X} \{*\}) \cong \bigsqcup_{x \in X} T(\{*\})$ which means $\# T(X) = \#T(\{*\}) \cdot \#X$. Then the hypothesis of Theorem \ref{thm:dimension_argument_strong} is fulfilled and therefore $\#T(X) = \#X$ or $\#T(X) = 2 \cdot \#X$, which shows that $\#T(\{*\})$ is either $\{*\}$ or $\{* , *\}$.
\end{proof}

We can further specify the result of Corollary \ref{cor:tangents_on_Finsetop}.b by showing that only one of these cases is possible. As a result we can show that $\mathrm{FinSet}^\mathrm{op}$ has no Cartesian tangent structure beyond the trivial one.

\begin{proposition}\label{prop:only_trivial_on_setop}
    Let $(\mathrm{FinSet}^\mathrm{op},T,p,+,0,\ell,c)$ be a Cartesian tangent structure on $\mathrm{FinSet}^\mathrm{op}$. Then it is in fact the trivial tangent structure i.e. $T=1_{\mathrm{FinSet}^\mathrm{op}}$, $p=+=0=\ell = c = 1$.
\end{proposition}
\begin{proof}
    According to Corollary \ref{cor:tangents_on_Finsetop}, there are two cases, either the tangent functor sends the one point set to two points $T(\{*\}) = \{0,1\}$ or $T(\{*\}) = \{*\}$.
    \begin{enumerate}
        \item[Case 1:] Suppose that $T(\{*\}) = \{0,1\}$. We will show that this is actually impossible and it leads to a contradiction. For this we look at the following diagrams 
        \begin{equation}\label{diag:addition_unitality}
        \begin{tikzcd}
    	{T(M)} && {T_2(M)} && {T(M)} && {T_2M} \\
    	&& {T(M)} &&&& {T(M)}
    	\arrow["{\langle0 \circ p, 1_{T(M)} \rangle}", from=1-1, to=1-3]
    	\arrow["{1_{T(M)}}"', from=1-1, to=2-3]
    	\arrow["{+}", from=1-3, to=2-3]
    	\arrow["{\langle 1_{T(M)} , 0 \circ p \rangle}", from=1-5, to=1-7]
    	\arrow["{1_{T(M)}}"', from=1-5, to=2-7]
    	\arrow["{+}", from=1-7, to=2-7]
        \end{tikzcd}
        \end{equation}
        which have to commute in any tangent category
        In $\FinSet^\mathrm{op}$ with $T(\{*\}) = \{0,1\}$, the pullback of $T$ along itself needs to have 3 elements (because cardinality is a dimension), we denote it as $T_2(*) = \{a,b,c\}$. The projection $p: T(\{*\}) \to \{*\}$ in $\FinSet^\mathrm{op}$ corresponds to a function $\{*\} \to \{0,1\}$ in $\FinSet$. Without loss of generality we assume it sends $*$ to $0$, i.e. the projection morphism is the inclusion of zero function $p=i_0$. The zero section $0 : \{*\} \to T(\{*\})$ starts at the initial object of $\FinSet^\mathrm{op}$ and thus is the unique such morphism. It corresponds to the unique function $\{0,1\} \to \{*\}$. 
        Using these insights we now write the diagrams \ref{diag:addition_unitality} as 
        \[\begin{tikzcd}
        	{\{0,1\}} &&& {\{a,b,c\}} & {\{0,1\}} &&& {\{a,b,c\}} \\
        	&&& {\{0,1\}} &&&& {\{0,1\}}
        	\arrow["{a \mapsto 0, b \mapsto0, c \mapsto1}"', from=1-4, to=1-1]
        	\arrow["{a \mapsto 0, b \mapsto 1, c \mapsto 0}"', from=1-8, to=1-5]
        	\arrow["{1_{\{0,1\}}}", from=2-4, to=1-1]
        	\arrow["{+}"', from=2-4, to=1-4]
        	\arrow["{1_{\{0,1\}}}", from=2-8, to=1-5]
        	\arrow["{+}"', from=2-8, to=1-8]
        \end{tikzcd}\]
        in $\FinSet$. The maps $\{a,b,c\} \to \{0,1\}$ are obtained because the unique morphisms into pullbacks correspond to unique morphisms out of pushouts.
        The first diagram only commutes if $+(1) = c$ and the second diagram only commutes if $+(1) = b$. Thus they can not both commute, a tangent structure with $T(\{*\}) = \{0,1\}$ is impossible
        \item[Case 2:] Suppose that $T(\{*\}) = \{*\}$. Then $p_{\{*\}}$, $+_{\{*\}}$, $0_{\{*\}}$, $\ell_{\{*\}}$ and $c_{\{*\}}$ are morphisms from $\{*\}$ to $\{*\}$. Since $\{*\}$ is initial in, the only such morphism is the identity. Now, since $(\FinSet^\mathrm{op},T,p,+,0,\ell,c)$ is Cartesian and every object is a product power (disjoint union) of $\{*\}$, the functor and structure morphisms also need to be the identity on every other object.
    \end{enumerate}
\end{proof}

As a consequence of the cardinality as a dimension on $\Set^\mathrm{op}$, the vertex set cardinality is also a dimension on $\Grph^\mathrm{op}$, the opposite category of looped graphs.

The category $\Grph$ of looped graphs has a forgetful functor $U : \Grph \to \Set$ given by taking the vertex set. This forgetful functor has a right-adjoint $R: \Set \to \Grph$ that sends a set to the complete graph on the set.
\begin{proposition}\label{prop:graph_vertex_number}~
\begin{enumerate}[label = (\alph*)]
    \item The assignment $v:\pi_0(\Grph) \to \mathbb N_\infty$ that sends a graph to its number of vertices is an $(\mathbb N_\infty, +)$-valued dimension on $\Grph^\mathrm{op}$.
    \item Let $(\Grph^\mathrm{op},T, p,0,+,\ell,c)$ be a tangent structure on $\Grph^\mathrm{op}$. Then for any graph $G$, the equation
    $$
    v(T^2(G))+ 2 \cdot v(G) = 3  \cdot v(T(G))
    $$
    holds.
    \item Let $(\Grph^\mathrm{op},T, p,0,+,\ell,c)$ be a tangent structure on $\Grph^\mathrm{op}$ and let $a \in \mathbb N$ be such that for every graph $G$, $v(T(G)) = a \cdot v(G)$. Then $a =1$ or $a=2$.
\end{enumerate}
\end{proposition}
\begin{proof}
\begin{enumerate}[label = (\alph*)]
\item Since $U:\Grph \to \Set$ is left-adjoint, $U^\mathrm{op}:\Grph^\mathrm{op} \to \Set^\mathrm{op}$ is right-adjoint. Now Corollary \ref{cor:adjunctions_full_subcats_and_dimensions} and Theorem \ref{thm:dimension_on_setop} produce the dimension on the opposite category of looped graphs described above.
\item This is a direct application of Theorem \ref{thm:dimension_result_weak}.
\item This is a direct application of Theorem \ref{thm:dimension_argument_strong}.
\end{enumerate}    
\end{proof}

An important extension of the opposite category of sets is the category of polynomial functors. Thus one could hope that the cardinality can be extended to a dimension on the category of polynomial functors.


\begin{example}
The category $\Poly$ of polynomial functors and natural transformations has been extensively analyzed in the literature, for example in \cite{spivak_book}.
For the precise definitions and constructions in $\Poly$ we ask the reader to read \cite{spivak_book}.
One of its key properties is that it contains $\Set^\mathrm{op}$ as a full subcategory, concretely as the full subcategory of monomials. The degree of a monomial corresponds to the cardinality of a set.

Since $\Poly$ is an extension of $\Set^\mathrm{op}$ it seems promising to define the degree of a polynomial functor as the dimension.
However, we will present an example of a pullback where the degree of polynomials does not add up like in Definition \ref{def:monoid-valued-dimension}.

Pullbacks in $\Poly$ are explained in \cite[Example 5.38]{spivak_book}. They are given by taking the pullback of monomial-sets and the pushout of exponents. Due to this characterization of pullbacks, the following square is a pullback in $\Poly$.
\[\begin{tikzcd}
	{y^{\{1,2,3\}}+ y^{\{a,b,c\}}} & {y^{\{3\}}+y^{\{a,b\}}} \\
	{y^{\{1,2\}}+ y^{\{c\}}} & {y^\emptyset+y^\emptyset}
	\arrow[from=1-1, to=1-2]
	\arrow[from=1-1, to=2-1]
	\arrow["{y^0 + y^0}", from=1-2, to=2-2]
	\arrow["{y^0 + y^0}"', from=2-1, to=2-2]
\end{tikzcd}\]
where for any set $X$, the map $0: \emptyset  \to X$ is the unique element of $\Hom_\Set(\emptyset,X)$ and $y: \Set^\mathrm{op} \to \Fun(\Set, \Set)$ is the Yoneda embedding. 
Since $0: \emptyset  \to X$ is a section in $\Set$, the natural transformations $y^0$ and $y^0 + y^0$ are retractions in $\Poly$.
Then
$$
\deg({y^{\{3\}}+y^{\{a,b\}}}) + \deg({y^{\{1,2\}}+ y^{\{c\}}}) = 2+2 = 4
$$
while
$$
\deg({y^{\{1,2,3\}}+ y^{\{a,b,c\}}}) + \deg(y^0 + y^0) = 3+0 =3.
$$
Thus
$$
\deg({y^{\{3\}}+y^{\{a,b\}}}) + \deg({y^{\{1,2\}}+ y^{\{c\}}}) \neq \deg({y^{\{1,2,3\}}+ y^{\{a,b,c\}}})
$$
showing that the polynomial functor degree is not a dimension on $\Poly$.
\end{example}

\section{Cardinality of groups as a dimension valued in a non-standard monoid}\label{sec:groups}
One of the most important algebraic structures on a set is a group. In \cite{ikonicoff2025abelianizationtangentcategories}, a tangent structure on the category of groups based on the abelianization of groups is defined. Like for sets the cardinality is the most important quantity to describe the size of a group. Therefore we will check below, if it is a dimension on $\Grp$ or $\Grp^\mathrm{op}$.

When thinking about tangent structures on groups, one may first think of the underlying manifold structure of Lie groups, which is a trivial tangent structure on discrete groups.
However, in \cite{ikonicoff2025abelianizationtangentcategories}, Sacha Ikonicoff, Jean-Simon Lemay and Tim Van der Linden showed that there is a different tangent structure on the category of groups that is based purely on algebraic properties, not topological properties of a group. We will explore the consequence of dimensions on this tangent structure in Example \ref{ex:tangent_structure_on_grp}.

An interesting aspect of the dimension on $\Grp$ is that it is valued in $(\mathbb N_\infty, \cdot)$ and not in $(\mathbb N_\infty, +)$. Unlike in all previous examples, the commutative monoid structure it is valued in is given by the multiplication, not the addition. This showcases that the use coming from being valued in any commutative monoid.
 
When considering the cardinality as a candidate for a dimension on $\Grp$, one needs to choose multiplication as the monoid operation to obtain an analogue for desired Property \ref{item:cartesian_products_to_sums} (additive w.r.t. products) of Remark \ref{rem:wishlist_dimension} as $\# (G \times G') = \#G \cdot \#G'$.

\begin{lemma}\label{lem:cardinality_groups_formula}
    Let $A \xrightarrow{f} C \xleftarrow{g} B$ be a cospan in \Grp{}, the category of groups, where $f$ is an epimorphism. Then the pullback $A \times_C B$ exists and
    $$
    \#(A \times_C B) \cdot \#C = \# A \cdot \#B . 
    $$
\end{lemma}
\begin{proof}
    The pullback exists since the category of finite groups is complete. Since the forgetful functor from finite groups to finite sets preserves limits, the underlying set of the pullback can explicitly be described by the formula
    $$
    A \times _C B = \{ (a,b) \in A \times B~ |~ f(a) = g(b)\}.
    $$
    We now rewrite this formula as 
    $$
    A \times_C B = \bigsqcup_{c \in \mathrm{Im}(g)} g^{-1}(c) \times f^{-1}(c) ,
    $$
    where $\mathrm{Im}(g)$ denotes the image and $g^{-1}(c)$ and $f^{-1}(c)$ denote pre-images. If $C$ and thus $\mathrm{Im}(g)$ is finite, it follows from the first isomorphism theorem and the surjectiveness of $f$ that for any $c \in \mathrm{Im}(g)$,
     this means that $\#g^{-1}(c) = \frac{\# B}{\# \mathrm{Im}(g)}$ and $\#f^{-1}(c) = \frac{\# A}{\# C}$. Therefore we conlude that 
    $$
    \# (A \times_C B) = \# \mathrm{Im}(g) \cdot \frac{\# B}{\# \mathrm{Im}(g)} \cdot \frac{\# A}{\# C} = \frac{\# B \cdot \#A}{\#C}.
    $$
    Multiplying both sides of the equation with $\# C$ concludes the proof for the case that $C$ is finite. 

    If $C$ is infinite, the surjectivity of $f$ implies that $A$ is also infinite. Thus 
    $$
    \#A \cdot \# B = \infty = \# C \cdot \# (A \times_B C).
    $$
    Here we use the convention of Definition \ref{def:monoids_N_infty}, that $0 \cdot \infty = \infty$.
\end{proof}
We observe that this formula looks very like Definition \ref{def:monoid-valued-dimension} with multiplication instead of addition. Therefore we will now consider the commutative monoid $(\mathbb N_\infty, \cdot)$, as defined in Definition \ref{def:monoids_N_infty}.

\begin{theorem}\label{thm:cardinality_groups}
The cardinality is 
\begin{enumerate}[label = (\alph*)]
    \item an $(\mathbb N_\infty , \cdot)$-valued dimension on $\Grp$ and
    \item an $(\mathbb N , \cdot)$-valued dimension on $\FinGrp$.
\end{enumerate}
\end{theorem}
\begin{proof}~
    \begin{enumerate}[label = (\alph*)]
    \item This is follows from Lemma \ref{lem:cardinality_groups_formula}.
    \item This follows from Corollary \ref{cor:adjunctions_full_subcats_and_dimensions} and the fact that $\FinGrp$ is a full subcategory of $\Grp$.
    \end{enumerate}
\end{proof}

\begin{corollary}\label{cor:cardinality_groups} 
Let $X$ be a group.
\begin{enumerate}[label = (\alph*)]
    \item For any tangent structure $(T,p,0,+,\ell,c)$ on $\Grp$ or $\FinGrp$, the equation
    $$
    \# T^2(X) \cdot (\#X)^2 = (\#T(X))^3
    $$ has to hold.
    \item For any tangent structure $(T,p,0,+,\ell,c)$ on $\FinGrp$ with $\#T(X) = (\#X)^n$, the constant $n$ can only be $n=1$ or $n=2$.
    \item The only tangent structure $(T,p,0,+,\ell,c)$ on $\FinGrp$ with $\#T(X) = n \cdot \#X$ fulfills $T(X) \cong X$.
\end{enumerate}
\end{corollary}
\begin{proof}~
    \begin{enumerate}[label = (\alph*)]
    \item This follows from Theorems \ref{thm:cardinality_groups} and \ref{thm:dimension_result_weak}. When formulated for the commutative monoid $(\mathbb N , \cdot)$ the commutative monoid operation that is denoted as $``+"$ in Theorem \ref{thm:dimension_result_weak} becomes the multiplication $``\cdot"$.
    \item Due to Part (a), the equation
    $$
    ((\#X)^n)^n \cdot (\#X)^2 = (\#X)^{3n}
    $$
    has to hold. For brevity, we denote $a:=\#X$ and simplify the equation to 
    $$
    a^{n^2+2} = a^{3n}.
    $$
    Since this has to hold for any finite group $X$, it has to hold for any $a \in \mathbb N$ and this implies that $n^2+2 = 3n$ which only holds for $n=1$ or $n=2$.
    \item Due to Part (a), the equation
    $$
    n^2 \cdot \#X \cdot (\#X)^2 = (n \cdot \#X)^3 
    $$
    has to hold. This implies $n^2 = n^3$, which is only true for $n=0$ and $n=1$. However, $n=0$, since every group has at least one element and thus $\#T(X) = 0$ is impossible. Thus $\#X = \#T(X)$. As a consequence of the first isomorphism theorem, the epimorphism $p:T(X) \to X$ that does not change the cardinality can only be an isomorphism.
\end{enumerate}
\end{proof}

As an example, we will now see how this formula applies to the known nontrivial tangent structure on \Grp.

\begin{example}[\cite{ikonicoff2025abelianizationtangentcategories}, Example 3.8]\label{ex:tangent_structure_on_grp}
    There is a tangent structure on \Grp{} whose tangent bundle endofunctor is given by 
    $$
    T(G) = G \times \mathrm{Ab}(G)
    $$
    where $\mathrm{Ab}(G)$ denotes the abelianization. This implies that 
    $$
    T^2(G) \cong G \times \mathrm{Ab}(G)\times \mathrm{Ab}(G) \times \mathrm{Ab}(G)
    $$
    because $\mathrm{Ab}(\mathrm{Ab}(G)) \cong \mathrm{Ab}(G)$.
    Define $n_G:=\#G$ and $k_G:=\#\mathrm{Ab}(G)$. 
    Since $\mathrm{Ab}(G)$ is a quotient of $G$, $k_G \leq n_G$ and $k_G$ divides $n_G$. 
    Then $\#T(G) = n_G \cdot k_G$ and $\#T^2(G) = n_G \cdot (k_G)^3$. Thus the equation
    $$
    \# T^2(G) \cdot (\#G)^2 = (\#T(G))^3
    $$
    becomes 
    $$
    n_G \cdot (k_G)^3 \cdot n_G^2 = (n_G \cdot k_G )^3.
    $$
    Thus, in this example, the equation from Corollary \ref{cor:cardinality_groups}(a) holds.
\end{example}
    In fact, Lemay and Ikonicoff have also shown a result on differential bundles in \cite[Example 4.16]{ikonicoff2025abelianizationtangentcategories}:  Every differential bundle is of the form $M \times A \xrightarrow{\pi_0} M$ for some abelian group $A$. This implies that 
    $$
    \frac{\# \mathrm{Ab}(M \times A)}{\#(M \times A)} = \#A = \frac{\# \mathrm{Ab}(M)}{\#M}.
    $$    
    We will now showcase the possibilities of a dimension by providing a different proof for this equation using the cardinality as a dimension.
\begin{proposition}
    Let $(\Grp, T, p,0,+,\ell)$ be the tangent structure from \cite{ikonicoff2025abelianizationtangentcategories} described in Example \ref{ex:tangent_structure_on_grp}. Let $E \to M$ be a differential bundle in $(\Grp, T, p,0,+,\ell)$. Then
    $$
    \frac{\# \mathrm{Ab}(E)}{\#E} = \frac{\# \mathrm{Ab}(M)}{\#M}
    $$
\end{proposition}
    \begin{proof}
        Due to Theorem \ref{thm:DiffBun_dimension_argument_weak}, the equation
        $$
            \#T(E)  \cdot (\#M)^2 = \#T(M) \cdot  (\#E)^2
        $$
        holds. Since $\#T(E) = \#E \cdot \# \mathrm{Ab}(E)$ and $\#T(M) = \#M \cdot \# \mathrm{Ab}(M)$, this can be expanded into
        $$
        \#\mathrm{Ab}(E) \cdot \#E  \cdot (\#M)^2 = \#\mathrm{Ab}(E) \cdot \#M \cdot (\#E)^2.
        $$
        Dividing both sides by $\#E^2 \cdot \# M^2$ leads to the equation
        $$
        \frac{\# \mathrm{Ab}(E)}{\#E} = \frac{\# \mathrm{Ab}(M)}{\#M}.
        $$
    \end{proof}

    In Example \ref{exs:vect_abGrp_as_examples_of_modules} we will further see that there is another nontrivial categorical dimension on the subcategory abelian groups, given by the rank over the integers. The reason for this is that every abelian group is a module over the integers.
Since the cardinality is an $(\mathbb N_\infty, + )$-valued dimension on $\Set^\mathrm{op}$ might suspect that the cardinality is also an $(\mathbb N_\infty, + )$-valued dimension on $\Grp^\mathrm{op}$, the opposite category of groups, since it was already a dimension on $\Set^\mathrm{op}$. However, Corollary \ref{cor:adjunctions_full_subcats_and_dimensions} does not apply here, since the forgetful functor $\Grp^\mathrm{op} \to \Set^\mathrm{op}$ does only have a right adjoint, not a left-adjoint.

In fact the cardinality is not a dimension on $\Grp^\mathrm{op}$. The reason is that the product in $\Grp^\mathrm{op}$ is the coproduct in $\Grp$, the free product of groups. The free product of groups does not fulfill desired Property \ref{item:cartesian_products_to_sums} (additive w.r.t. products) of Remark \ref{rem:wishlist_dimension}. For finite groups $A$ and $B$, the cardinality of $\#(A \otimes B)$ can even be infinite.

\section{Characteristic, Krull dimension and the lowest common multiple}\label{sec:rings}
In this section we investigate different approaches to dimensions of rings. First we define several different categories of rings.  The different choices we consider depend on whether or not the ring is commutative, whether or not the ring is unital, and whether or not the homomorphisms must preserve the unit.

Then we will show that the characteristic of a unital ring is a categorical dimension. This will imply that any tangent structure on the category of unital rings must satisfy certain constraints. This is part way of showing that any tangent structure on rings has to preserve the characteristic.

 Then we will show that the Krull dimension is neither a categorical dimension on the category of rings, nor on its opposite. This is surprising since the opposite category of commutative unital rings is equivalent to the category of affine schemes and the dimension of the scheme corresponds to the Krull dimension of the ring. However, the dimension of an affine scheme is only compatible with the tangent structure for smooth schemes. We suspect the Krull dimension will be a categorical dimension on the opposite of the subcategory of commutative unital rings which corresponds to smooth schemes.

Let $\Ring_{n}$ be the category of all rings (possibly without unit, the $n$ stands for non-unital) and all ring homomorphisms (functions that preserve zero, addition and multiplication). 
Let $\Ring_1$ denote the full subcategory of $\Ring$ consisting of all unital rings, i.e. rings with a unit element. The morphisms may not preserve the unit.
Let $\Ring_u$ denote the subcategory of $\Ring_1$ that only allows unit-preserving homomorphisms.
We define $\CRing_n$, $\CRing_1$ and $\CRing_u$ as the full subcategories of $\Ring_n$, $\Ring_1$ and $\Ring_u$ where the objects are commutative rings. For overview, we summarize this in Table \ref{tab:different_ring_cats}. 
\begin{table}[h]
    \centering
        \begin{tabular}{c|c|c}
             Name & Objects & Morphisms \\
             \hline
            $\Ring_n$ & any rings & any ring homomorphisms \\
            $\Ring_1$ & unital rings & any ring homomorphisms \\
            $\Ring_u$ & unital rings & unit preserving ring homomorphisms \\
            $\CRing_n$ & any commutative rings & any ring homomorphisms \\
            $\CRing_1$ & unital commutative rings & any ring homomorphisms \\
            $\CRing_u$ & unital commutative rings & unit preserving ring homomorphisms
        \end{tabular}
    \caption{The conventions we use for the different categories of rings}
    \label{tab:different_ring_cats}
\end{table}

In other literature, $\Ring$ sometimes denotes $\Ring_n$ and sometimes $\Ring_u$ while $\CRing$ most often denotes $\CRing_u$. For clarity we do not use undecorated $\Ring$ and $\CRing$.

Before we define dimensions on the categories of rings, we will first describe pullbacks of rings.
\begin{lemma}\label{lem:pullbacks_of_rings}
    Any pullback
    \[\begin{tikzcd}
    	{A \times_B C} & A \\
    	C & B
    	\arrow[from=1-1, to=1-2]
    	\arrow[from=1-1, to=2-1]
    	\arrow["f", from=1-2, to=2-2]
    	\arrow["g"', from=2-1, to=2-2]
    \end{tikzcd}\]
    in $\Ring_n$, $\Ring_1$, $\Ring_u$, $\CRing_n$, $\CRing_1$ or $\CRing_u$, is of the form
    $$
    A \times_B C = \{(a,b) \in A \times B ~|~f(a) = g(b) \},
    $$
    if it exists. The ring operations are defined componentwise.
\end{lemma}
\begin{proof}
    The formula for $\Ring_n$ is an easy application of the definition of a limit.
    The formula for $\Ring_1$, $\CRing_n$ and $\CRing_u$ follows from the fact that the embedding of full subcategories preserves limits. The formula for $\Ring_u$ is a direct application of the definition of a limit and the formula for $\CRing_u$ follows from the fact that $\CRing_u$ is a a full subcategory of $\Ring_u$. 
\end{proof}

Pullbacks may not always exist in $\Ring_1$ and $\CRing_1$. An example for this is the pullback
\[\begin{tikzcd}
	{2 \mathbb Z} && { \mathbb Z} \\
	{\mathbb Z} && {\mathbb Z  \times \mathbb Z/2}
	\arrow[hook, from=1-1, to=1-3]
	\arrow[hook, from=1-1, to=2-1]
	\arrow["\lrcorner"{anchor=center, pos=0.125}, draw=none, from=1-1, to=2-3]
	\arrow["{x \mapsto(x,0)}", from=1-3, to=2-3]
	\arrow["{x \mapsto (x,[x])}"', from=2-1, to=2-3]
\end{tikzcd}\]
in $\Ring_n$ which does not exist in $\Ring_1$ since $2 \mathbb Z$, the set of even integers, has no multiplicative unit. 

\subsection{Characteristic as a dimension valued in the least common multiple monoid}
In this section we will use a monoid structure on the natural numbers which is neither of the operations on the rig N. At first one might think that the characteristic of a ring is a strange choice for a dimension, since in general $\mathrm{char}(R \times R') \neq \mathrm{char}(R) +\mathrm{char}(R') $ and therefore desirable Property \ref{item:cartesian_products_to_sums} (additive w.r.t. products) does not hold. 
However, the least common multiple is also a monoid structure on $\mathbb N_\infty^*$. 
\begin{definition}
    Let $\mathbb N^*_\infty$ denote the set of nonzero natural numbers and infinity
    $$
    \mathbb N^*_\infty := \mathbb N \setminus \{0\} \sqcup \{ \infty \} = \{1,2,3,...,\infty\}.
    $$
    Given $a,b \in \mathbb N_\infty^*$, we define their \textbf{least common multiple} as
    $$
    \lcm(a,b) = \min(\{ x \in \mathbb N_\infty^* ~|~ \exists s,t \in \mathbb N_\infty^* , x = s \cdot a = t \cdot b \}).
    $$
    We define $\lcm(a,\infty) = \lcm(\infty, b) = \infty$ for all $a,b \in \mathbb N_\infty\setminus \{0\}$.
\end{definition}
The least common multiple is a commutative monoid with unit $1$. In particular $\lcm$ is associative, since the least common multiple of three numbers $a,b,c$ is both equal to $\lcm(a,\lcm(b,c))$ and to $\lcm(\lcm(a,b),c)$.

The characteristic is additive with respect to this monoid structure, i.e.
$$\mathrm{char}(R \times R') = \lcm (\mathrm{char}(R),\mathrm{char}(R')),$$and thus Property \ref{item:cartesian_products_to_sums} (additive w.r.t. products) holds with respect to this monoid structure.

In this section we will show that the characteristic of a unital ring is a dimension. The characteristic is defined as follows.
\begin{definition}
    Let $R$ be a ring. The characteristic $\mathrm{char}(R)$ is defined as
    $$
    \mathrm{char}(R) = \min\{ n \in \mathbb N\setminus\{0\} ~|~ n \cdot x =0 ~,\forall x \in R \}.
    $$
    If there is no number $n \in \mathbb N$ such that  for all $x \in R$, $n \cdot x=0$, we say the characteristic is $\infty$. 
    
    If the characteristic is finite, this is the classical definition in \cite{Gallian_contemporary_abstract_algebra}. Classically, the characteristic is defined to be zero if there is no number $n \in \mathbb N$ such that $n \cdot x=0$ for all $x$.
\end{definition}
In the literature it is more common to call the characteristic zero when we call it infinity. 
We use this unusual convention since it will make it easier to define a monoid structure such that the characteristic is a categorical dimension.

\begin{lemma}\label{lem:maximal_order_elem}
Let $(R,+,\cdot,0)$ be a ring of characteristic $n$. Then there exists an element $x_R \in R$ such that $k \cdot x_R \neq 0$ for all $k < n$. We call this element a \textbf{maximal order element}.
\end{lemma}
\begin{proof}
    The key insight for this proof is that the mulitplicative structure of $R$ does not matter this is just a statement about the abelian group $(R,+,0)$.
    Due to the characterization of finite exponent abelian groups in \cite[Theorem 8.6]{Kaplansky_infinite_abelian_groups},
    the abelian group $(R,+,0)$ is isomorphic (as a group) to a product of cyclic groups.
    $$
    R \overset{\varphi}{\cong} \prod_{i \in i} \mathbb Z/n_i\mathbb Z    
    $$
    where $n$ is the least common multiple of the $n_i$.
    The element $x_R = \varphi^{-1}(1,1,...)$ fulfills the required property since $k \cdot (1,1,...) \neq 0$ for $k <n$.
\end{proof}

Section retraction pairs have a key role in the definition of categorical dimensions and in the definition of tangent categories. Thus we will frequently reason about characteristics using section retraction pairs. For this, it is good to keep in mind that sections can not decrease characteristic.
\begin{lemma}\label{lem:characteristic_section_retraction}
    Given a section $A \xrightarrow{s} B $ with retraction $ B \xrightarrow{r} A$ in $\Ring_n$, the inequality
    $$
    \mathrm{char}(A) \leq \mathrm{char}(B)
    $$
    holds. Additionally, $\mathrm{char}(B)$ is a multiple of $\mathrm{char}(A)$.
\end{lemma}
\begin{proof}
    We distinguish the cases when $A$ has finite characteristic and when $A$ has infinite characteristic.
    
    \textit{Case 1}: $A$ has infinite characteristic. 
    Since $A$ has infinite characteristic, for every $n \in \mathbb N\setminus\{0\}$, there is an $x_n \in A$ such that $n \cdot x_n \neq 0$. Then $s(x_n) \in B$ fulfills that 
    $$
    n \cdot x_n = n \cdot r(s(x_n)) = r (n \cdot s(x_n)) \neq 0.
    $$
    Thus $n \cdot s(x_n) \neq 0$. Since $n$ was arbitrary, this proves that $B$ has infinite characteristic.

    \textit{Case 2}: $A$ has finite characteristic. Let $x_A$ be the element constructed in Lemma \ref{lem:maximal_order_elem}.
    Since $r \circ s=1_A$, for every $n < \mathrm{char}(A)$, we know that 
    $$
    0\neq n \cdot r(s(x_A)) = r(n \cdot s(x_A)).
    $$
    Since $r$ is a ring-homomorphism, this implies for $s(x_A) \in B$ that $n \cdot s(x_A) \neq 0$ implying that $\mathrm{char}(B) \geq \mathrm{char}(A)$.

    On the other hand, 
    $$
    s(\mathrm{char(B)} \cdot x_A) =  \mathrm{char}(B) \cdot s(x_A) = 0
    $$
    and therefore 
    $$
    \mathrm{char(B)} \cdot x_A = r(s(\mathrm{char(B)} \cdot x_A)) = r(0) = 0
    $$
    which implies that $\mathrm{char(B)}$ is a multiple of $\mathrm{char}(A)$.  
\end{proof}

\begin{proposition}\label{prop:characteristic_as_dim_on_Ring}
    The characteristic $\mathrm{char}$ is an $(\mathbb N_\infty^*, \lcm )$-valued dimension on $\Ring_n$, $\Ring_1$ and on $\Ring_u$. 
\end{proposition}
\begin{proof}
    We will first check the criteria of Definition \ref{def:monoid-valued-dimension} for $\Ring_u$ and then for $\Ring_1$.

    In $\Ring_u$, given any morphism $f: R \to R'$, the characteristic can not increase $\mathrm{char}(R')\leq \mathrm{char}(R) $ since $ 0 = f(0) = f(\mathrm{char}(R) \cdot 1)$. Thus, given a section retraction pair between $R$ and $R'$, $\mathrm{char}(R) = \mathrm{char}(R')$. Now let 
    \[\begin{tikzcd}
    	{A \times_B C} & A \\
    	C & B
    	\arrow[from=1-1, to=1-2]
    	\arrow[from=1-1, to=2-1]
    	\arrow["{r}", from=1-2, to=2-2]
    	\arrow["{f}"', from=2-1, to=2-2]
    \end{tikzcd}\]
    be any of the pullbacks from Definition \ref{def:monoid-valued-dimension} in $\Ring_u$.
    Then $f$ and $r$ are both parts of section-retraction pairs and the characteristic of all 4 vertices is the same. Therefore
    $$
    \lcm(\mathrm{char}(A), \mathrm{char}(C)) = \lcm(\mathrm{char}(A \times_B C) , \mathrm{char}(B))
    $$
    holds trivially.

    Now let 
    \[\begin{tikzcd}
    	{A \times_B C} & A \\
    	C & B
    	\arrow[from=1-1, to=1-2]
    	\arrow[from=1-1, to=2-1]
    	\arrow["{r}", from=1-2, to=2-2]
    	\arrow["{f}"', from=2-1, to=2-2]
    \end{tikzcd}\]
    be any of the pullbacks from Definition \ref{def:tangent_dimension} in $\Ring_n$. 
    Assuming that $A$ and $C$ have finite characteristic, let $x_A \in A$ and $x_C \in C$ be the maximal order elements constructed in Lemma \ref{lem:maximal_order_elem}.
    We will distinguish the cases when $f$ is a section and when $f$ is a retraction.

    \textit{Case 1:} Suppose $f$ is a retraction, in particular it is surjective. Then there is a $y_C \in C$ such that $f(y_c) = r(x_A)$ and thus we have $(x_a, y_C) \in A \times_B C$ such that for $k\leq \mathrm{char}(A)$, $k \cdot (x_A, y_C) = (k \cdot x_A , k \cdot y_C) \neq 0$. Thus $\mathrm{char}(A \times_B C)$ is a multiple of $\mathrm{char}(A)$.

    Analogously there is an $y_A \in A$ such that $(y_A, x_C) \in A \times_B C$ and thus $\mathrm{char}(A \times_B C)$ is a multiple of $\mathrm{char}(C)$. This shows that $\mathrm{char}(A \times_B C)$ is a multiple of $\lcm(\mathrm{char}(A), \mathrm{char}(C))$.
    Since $A \times_B C$ is a subset of $A \times C$, the equation
    $$
    \mathrm{char}(A \times_B C) = \lcm(\mathrm{char}(A), \mathrm{char}(C))
    $$
    holds. Due to Lemma \ref{lem:characteristic_section_retraction}, $\mathrm{char}(B)$ divides $\mathrm{char}(A)$, implying that
    $$
    \lcm(\mathrm{char}(A \times_B C), \mathrm{char(B)}) = \lcm( \lcm(\mathrm{char}(A), \mathrm{char}(C)),\mathrm{char}(B)) = \lcm(\mathrm{char}(A), \mathrm{char}(C)).
    $$
    This is exactly the equation we needed to show for Definition \ref{def:monoid-valued-dimension}.

    \textit{Case 2:} Suppose $f$ is a section. Then, due to Lemma \ref{lem:characteristic_section_retraction}, $\mathrm{char}(B)$ and $\mathrm{char}(C)$ both divide $\mathrm{char}(A)$. 
    
    Let $s: B \to A$ be a section for the retraction $r$. We can express the maximal order element $x_A \in A$ as $x_A = s(r(x_a))+ z_A$ with $r(z_A) = 0$, by choosing $z_A = x_A - s(r(x_A))$.

    For $k = \lcm(\mathrm{char}(A \times_B C), \mathrm{char}(B))$, we observe that 
    $$
    k \cdot x_A = k \cdot s(r(x_A)) + k \cdot z_A.
    $$
    Since $k$ is a multiple of $\mathrm{char}(B)$, $k \cdot s(r(x_A)) = s(k \cdot r(x_A)) = 0$. 
    Since $r(z_A) = 0$, we have the element $(z_A ,0) \in A \times_B C$ and since $k$ is a mutiple of $\mathrm{char}(A \times_B C)$, $k \cdot (z_A,0) = 0$, implying that $k \cdot z_A =0$.

    Thus $k \cdot x_A = 0$ and since $x_A$ had maximal order this shows that $\mathrm{char}(A)$ divides $\lcm(\mathrm{char}(A \times_B C), \mathrm{char}(B))$. Since $\mathrm{char}(C)$ divides $\mathrm{char}(A)$, the equation $\lcm(\mathrm{char}(A), \mathrm{char}(C)) = \mathrm{char}(A)$ holds, implying that $\lcm(\mathrm{char}(A), \mathrm{char}(C)) = \mathrm{char}(A)$ divides $\lcm(\mathrm{char}(A \times_B C), \mathrm{char}(B))$. Since $A \times_B C$ is a subset of $A \times C$, this shows that 
    $$
    \lcm(\mathrm{char}(A \times_B C), \mathrm{char}(B)) = \lcm(\mathrm{char}(A), \mathrm{char}(C)).
    $$
    This is exactly the equation we needed to show for Definition \ref{def:monoid-valued-dimension}.

    Now it only remains to show that the equation 
    $$
    \lcm(\mathrm{char}(A \times_B C), \mathrm{char}(B)) = \lcm(\mathrm{char}(A), \mathrm{char}(C)).
    $$
    also holds when $\mathrm{char}(A)$ or $\mathrm{char}(C)$ are infinite. If $\mathrm{char}(B)=\infty$, it holds trivially. Suppose $\mathrm{char}(B)$ is finite. Since in the case that $f$ is a section $\mathrm{char}(C) \leq \mathrm{char}(B)$, and in the case that $f$ is a retraction, $C$ and $A$ can be swapped, we can assume without loss of generality that $\mathrm{char}(A) = \infty$. 

    Let $n \in \mathbb N\setminus\{0\} $ be arbitrary. Let $x_n \in A$ be such that $n \cdot \mathrm{char}(B) \cdot x_n \neq 0$. Then $r(\mathrm{char}(B) \cdot x_n) = 0$ and thus $(\mathrm{char}(B) \cdot x_n, 0) \in A \times_B C$ fulfills that 
    $$
    n \cdot (\mathrm{char}(B) \cdot x_n, 0) \neq 0.
    $$
    Since $n$ was arbitrary, this shows that $\mathrm{char}(A \times_B C) = \infty$ and thus 
    $$
    \lcm(\mathrm{char}(A \times_B C), \mathrm{char}(B)) = \lcm(\mathrm{char}(A), \mathrm{char}(C)).
    $$
    holds. This concludes the proof that $\mathrm{char}$ is a $(\mathbb N^*_\infty, \lcm)$-valued dimension on $\Ring_n$. Since $\Ring_1$ is a full subcategory of $\Ring_n$, Corollary \ref{cor:adjunctions_full_subcats_and_dimensions} implies that it also is a $(\mathbb N^*_\infty, \lcm)$-valued dimension on $\Ring_1$.
\end{proof}
\begin{corollary}\label{cor:characteristic_as_dim_on_CRing}
    The characteristic $\mathrm{char}$ is an $(\mathbb N_\infty^*, \lcm )$-valued dimension on $\CRing_1$ and on $\CRing_u$. 
\end{corollary}
\begin{proof}
    This follows directly from Proposition \ref{prop:characteristic_as_dim_on_Ring} and Corollary \ref{cor:adjunctions_full_subcats_and_dimensions}.
\end{proof}

As a consequence, the characteristic of the tangent bundle is now determined in $\Ring_1$ and $\Ring_u$. However, the characteristic of a tangent bundle in $\Ring_u$ is determined anyways.
\begin{proposition}
    Let $(T,p,0,+, \ell , c)$ be a tangent structure on $\Ring_u$. Then for any unital ring $R$, 
    $$
    \mathrm{char}(T(R)) = \mathrm{char}(R).
    $$
\end{proposition}
\begin{proof}
    This simply follows from the fact that any unit-preserving ring homomorphism can only decrease characteristic. Since both $R \xrightarrow{0} T(R)$ and $T(R) \xrightarrow{p} R$ are unit-preserving ring homomorphisms, the characteristics have to equal.
\end{proof}
For $\Ring_1$, however the dimension brings an additional insight.
\begin{corollary}\label{cor:char_of_t2_lcm}
        Let $(T,p,0,+, \ell , c)$ be a tangent structure on $\Ring_n$ or $\Ring_1$.
        Then for every object $R$,
        $$
        \lcm (\mathrm{char}(T^2(R)) , \mathrm{char}(R) ) = \mathrm{char}(T(R)) 
        $$
\end{corollary}
\begin{proof}
    This is Theorem \ref{thm:dimension_result_weak} applied to the dimension defined in Proposition \ref{prop:characteristic_as_dim_on_Ring}.
\end{proof}
However, since we have a section retraction pair $R \xrightarrow{0} T(R) , T(R) \xrightarrow{p} R$, we can even conclude more.
\begin{corollary}\label{cor:tangent_structures_on_Ring_n}
    Let $(T,p,0,+, \ell , c)$ be a tangent structure on $\Ring_n$ or $\Ring_1$.
        Then for every object $R$ there is a number $n_R$ such that
        $$
         \mathrm{char} (T(R)) = n_R \cdot \mathrm{char} (R)
        $$
        and
        $$
        \mathrm{char}(T^2(R)) = \mathrm{char} (T(R)).
        $$
\end{corollary}
\begin{proof}
    Corollary \ref{cor:char_of_t2_lcm} implies that the characteristic of $T(R)$ is a multiple of the characteristic of $R$.
    
    Corollary \ref{cor:char_of_t2_lcm} also implies that it is a multiple of the characteristic of $T^2(R)$.
    Due to the same argument applied to $T(R)$ instead of $R$, $\mathrm{char}(T^2(R))$ is a multiple of $\mathrm{char}(T(R))$. This implies that $\mathrm{char}(T(R)) = \mathrm{char}(T^2(R))$.
\end{proof}

In \cite{rosicky} an example of a tangent structure on $\CRing_u$ is constructed.
\begin{example}\label{ex:tangent_on_ring}\cite[Example 2]{rosicky}
    For the category $\CRing_u$ of unital commutative rings and unit-preserving ring homomorphisms, the there is a tangent structure on $\CRing_u$ where
    the tangent bundle endofunctor $T$ sends a ring $R$ to $R[x]/(x^2)$.
\end{example}
Indeed, this tangent bundle endofunctor preserves the characteristic: $\mathrm{char}(T(R)) = \mathrm{char}(R)$.

All pullbacks of $\CRing_u$ are squares of ring homomorphisms that are pullbacks of underlying sets. Thus all pullbacks of $\CRing_u$ are also pullbacks in $\CRing_1$, implying that any tangent structure on $\CRing_u$ is also tangent structure on $\CRing_1$. Thus \ref{ex:tangent_on_ring} is also an example of a tangent structure on $\CRing_1$ that preserves the characteristic.

It is unclear if an example of a tangent structure on $\Ring_n$, $\CRing_n$, $\Ring_1$ or $\CRing_1$  exists where the tangent bundle endofunctor does not preserve the characteristic. If it exists, it needs to fulfill the requirement of Corollary \ref{cor:tangent_structures_on_Ring_n}. 

\subsection{Krull dimension satisfying products but not pullbacks}
The term ``dimension of a ring" is commonly used for the Krull dimension of a commutative ring, which is defined through ascending chains of prime ideals and has extensive applications in algebraic geometry. Since 
$$
\dim(R \times R') = \max(\dim(R), \dim(R')),
$$
desirable property \ref{item:cartesian_products_to_sums} (additive w.r.t. products) of Remark \ref{rem:wishlist_dimension} holds when choosing $\max$ as the monoid structure.

However, we will show that the pullback condition of Definition \ref{def:monoid-valued-dimension} does not hold and it thus is not a dimension on $\CRing_u$.
\begin{definition}\label{def:Krull_dimension}\cite[Page 217]{Eisenbud:CommutativeAlgebrawithaViewTowardAlgebraicGeometry}
    Given a commutative ring $R$ with unit, its Krull dimension is the length $n$ of the longest chain of prime ideals
    $$
    0 \subsetneq p_1 \subsetneq ... \subsetneq p_n \subsetneq R.
    $$
\end{definition}
The Krull dimension fulfills some useful properties.
    \begin{enumerate}[label =(\alph*)]
        \item Let $F$ be a field. Then 
        $
        \dim (F[x_1, ... ,x_n]) = n.
        $
        \item Let $R, R'$ be commutative rings then $\dim(R \times R') = \max(\dim(R), \dim(R'))$.
    \end{enumerate}

We show here that this is not an $\mathbb N_\infty$ valued dimension on $\CRing_u$ in the sense of Definition \ref{def:monoid-valued-dimension}.
The following examples will show that it is a dimension valued in none of the monoid structures on $\mathbb N_\infty$ defined in \ref{def:monoids_N_infty}.


The following example will now show that it can not be an $\mathbb (N_\infty , \max)$-valued dimension.
\begin{example}\label{ex:Krull_not_max}
    The diagram 
\begin{equation}
\begin{tikzcd}
{\mathbb Q [y]} & {\mathbb Q[x,y] / (x y)} \\
	{\mathbb Q} & {\mathbb Q[x]}
	\arrow[from=1-1, to=1-2]
	\arrow[from=1-1, to=2-1]
	\arrow["\lrcorner"{anchor=center, pos=0.125}, draw=none, from=1-1, to=2-2]
	\arrow[two heads, from=1-2, to=2-2]
	\arrow[hook, from=2-1, to=2-2]
\end{tikzcd}\label{diag:Krull_not_max}
\end{equation}
    is a pullback in $\CRing_u$, where the map $\mathbb Q \hookrightarrow \mathbb Q[x]$ includes a number as the constant term and the map $\mathbb Q[x,y]/(xy) \to \mathbb Q [x]$ evaluates at $y=0$. It is a pullback because giving a pair $(r,p)$ of $r \in \mathbb Q$ and $p\in \mathbb Q[x,y]/(xy)$ such that $p(x,0) = r$ is the same as giving a polynomial $p \in \mathbb C[x,y]/(xy)$ with no $x$-terms, i.e. a polynomial $p \in \mathbb C[y]$.
\end{example}
The very same example shows that it is not a categorical dimension on $\CRing_1$.

For a field $F$, the coproduct $F[x_1 ,...,x_n] \otimes F[x_1 ,...,x_k] = F[x_1 , ... , x_{n+k}]$. Since $\dim(F[x_1 ,...,x_n]) = n$, this suggests that for the opposite category $\CRing_u^\mathrm{op}$ desirable property \ref{item:cartesian_products_to_sums} (additive w.r.t. products) from Remark \ref{rem:wishlist_dimension} holds for  $(\mathbb N_\infty , +)$.

The context of algebraic geometry would also suggest that the Krull dimension is a categorical dimension on $\CRing_u^\mathrm{op}$. The category $\CRing_u^\mathrm{op}$ is equivalent to the category of affine schemes and for smooth affine schemes the scheme-theoretical dimension equals the Krull dimension. However, using the counterexample in Example \ref{ex:not_krull_on_ring_op} that corresponds to a pullback of non-smooth schemes, we will show that it still is not a Dimension in the sense of Definition \ref{def:monoid-valued-dimension}. 

\begin{example}\label{ex:not_krull_on_ring_op}
    Let $\mathbb Q$ denote the field of rational numbers. Let $i'$ denote the embedding $\mathbb Q[x] \hookrightarrow \frac{\mathbb Q[x,y]}{\langle x \cdot y\rangle}$. The diagram
    \[\begin{tikzcd}
    	{\mathbb Q[x,y] \over \langle x \cdot y\rangle} && {\mathbb Q[x,y,z] \over \langle x \cdot y\rangle} \\
    	{\mathbb Q[x]} && {\mathbb Q[x,y,z] \over \langle x \cdot y\rangle}
    	\arrow["i", hook, from=1-1, to=1-3]
    	\arrow["{\mathrm{ev_0}}"', two heads, from=1-1, to=2-1]
    	\arrow["1", from=1-3, to=2-3]
    	\arrow["{i \circ i'}"', from=2-1, to=2-3]
    \end{tikzcd}\]
     is a pushout square in $\CRing$ and thus corresponds to a pullback square in $\CRing^\mathrm{op}$. However, the Krulld dimensions of the rings involved are
     $$
     \dim(\mathbb Q[x]) = 1 ~,~ \dim\left({\mathbb Q[x,y,z] \over \langle x \cdot y\rangle}\right) = 3 ~,~ \dim\left({\mathbb Q[x,y] \over \langle x \cdot y\rangle}\right)=2
     $$
     and thus
     $$
    \dim\left({\mathbb Q[x,y,z] \over \langle x \cdot y\rangle}\right) + \dim(\mathbb Q[x]) = 3+1 \neq 3+2 = \dim\left({\mathbb Q[x,y,z] \over \langle x \cdot y\rangle}\right) + \dim\left({\mathbb Q[x,y] \over \langle x \cdot y\rangle}\right).
    $$
\end{example}

\begin{remark}
    The category of affine schemes is equivalent to the opposite category of $\CRing_u$ and the Krull dimension corresponds to the dimension of a scheme. However, the dimension of a scheme is only well-behaved with respect to the tangent structure for smooth schemes over a field, as one can read in \cite[Chapter 13]{Vakil_rising_sea}.
    
    Thus the geometric intuition suggests that the Krull dimension is an $(\mathbb N,+)$-valued dimension in the sense of Definition \ref{def:monoid-valued-dimension} on the smooth objects of $\CRing_u^\mathrm{op}$. We suspect this is true, but the proof will require technical details from algebraic geometry which are beyond the scope of this paper. 
\end{remark}

\section{Transporting a dimension between modules and algebras with an adjunction}\label{sec:modules}

In this section we will give an example of a categorical dimension on the category of $R$-modules and transport this dimension to the category of $R$-algebras using an andjunction.

First, we will show that the rank of an $R$-modules is a categorical dimension on the category $\Mod_R$ of $R$-modules and $R$-modules homomorphisms for any PID, $R$. In particular, this shows that the dimension of a vector space is a categorical dimension on the category $\Vect_F$ for a field, $F$. As a consequence, a tangent structure on $\Mod_R$ must fullfil the equation in Corollary \ref{cor:tangent_structures_on_Mod}.

Next, we show that using the free-forgetful adjunction between $R$-algebras and $R$-modules, we obtain a category dimension on $R$-algebras.  Finally, as an additional observation, we use the categorical dimension on $\Vect$ and show that there can only be two tangent structures on $\Vect$, and we identify them.


In \cite{maclane:71}[Section IV.2] Mac Lane shows that there is a free forgetful adjunction between $\Mod_R$ and $\AbGrp$. Thus, by Corollary \ref{cor:adjunctions_full_subcats_and_dimensions}, the cardinality of the underlying abelian group is already an $(\mathbb N_\infty, \cdot)$-valued dimension on $\Mod_R$. In Section \ref{subsec:rank_of_a_module} we will show that the rank is a second categorical dimension on $\Mod_R$.

\begin{example}\label{ex:degenerate_module_tangent}
The category $\Mod_R$ has a somewhat degenerate tangent structure given for an $R$-module $V$ by
\begin{itemize}
    \item the tangent bundle endofunctor $T(V) = V \times V$ with $T(f) = f \times f$,
    \item the projection $p_V = \pi_0: V \times V \to V$,
    \item the zero section $0_V = \langle 1_V, \mathrm{const}_0 \rangle: V \to V \times V$,
    \item the addition $+_V = \langle \pi_0, \mathrm{sum}(\pi_1, \pi_2)\rangle $, where $\mathrm{sum}(v,w) = v+w$,
    \item the vertical lift $\ell_V = \langle \pi_0, \mathrm{const}_0 , \mathrm{const}_0, \pi_1 \rangle : V \times V \to V \times V \times V \times V $
    \item the canonical flip $c_V = \langle \pi_0 , \pi_2, \pi_1 , \pi_3 \rangle : V \times V \times V \times V \to V \times V \times V \times V$.
\end{itemize}
These maps fulfill all required equations and are natural transformations since all morphisms are $R$-linear. 
The universality of the vertical lift holds since the diagram
\[\begin{tikzcd}
	{V \times V \times V} && {V \times V \times V \times V} \\
	V && {V \times V}
	\arrow["{\langle \pi_0  , \pi_2, \mathrm{const}_0 ,\pi_1 \rangle}", from=1-1, to=1-3]
	\arrow["{\pi_0}"', from=1-1, to=2-1]
	\arrow["{\langle \pi_0 , \pi_2\rangle}", from=1-3, to=2-3]
	\arrow["{\langle 1_V,\mathrm{const}_0\rangle}"', from=2-1, to=2-3]
\end{tikzcd}\]
is a pullback square.    
\end{example}

\subsection{Rank of a module over a ring}\label{subsec:rank_of_a_module}
The rank of is the most common generalization of the dimension of vector spaces to modules. 
A module's rank is a natural number or infinity (like in desired Property \ref{item:dim_is_natural_number} (valued in $\mathbb N$) of Remark \ref{rem:wishlist_dimension}), invariant under isomorphism (like in desired Property \ref{item:isomorphism} (isomorphism invariant) of Remark \ref{rem:wishlist_dimension}) and it is additive under cartesian products (like in desired Property \ref{item:cartesian_products_to_sums} (additive w.r.t. products) of Remark \ref{rem:wishlist_dimension}).
Desired Property \ref{item:embeddings_inequalities} (embeddings to inequalities) holds since the rank of a submodule is less than or equal to the rank of the module itself.
Desired Property \ref{item:locality}(the same as locally) does not make sense to ask, because an open subset of a module is in general not a module itself. 

We will show below that it is indeed an $(\mathbb N_\infty, +)$-valued dimension in the sense of Definition \ref{def:monoid-valued-dimension}, so the generalized version of desired property \ref{item:pullbacks} (additive w.r.t. pullbacks) holds. 

Recall that a \textbf{principal ideal domain} (PID) is a unital commutative ring without zero divisors where every ideal is generated by a single element.
Throughout this section we fix a principal ideal domain $R$. 
\begin{definition}
    Given a $R$-module $A$. Then \textbf{the rank} $\mathrm{rk}(A) \in \mathbb N_\infty$ is the cardinality of any maximal linear independent subset of $A$. If there is no finite maximal linearly independent subset of $A$, the rank is $\rk(A) = \infty$.
\end{definition}
It is a standard result, for example in \cite{Aluffi}, that any maximal linear independent subset of $A$ has the same cardinality.
A key property of the rank is its behaviour under short exact sequences.

\begin{lemma}\label{lem:module_rank_for_short_exact_sequence_PID_finitely_generated}
    Let $R$ be a PID and let
    $$
    0 \to A \to B \to C \to 0
    $$
    be a short exact sequence of $R$-modules with finite rank. Then $\rk(A) + \rk(C) = \rk(B)$   
\end{lemma}
This is a standard result, for example Exercise 3.14 in chapter VI of \cite{Aluffi}, thus we omit the proof here.

 Actually, as we show below, the requirement that $A,B$ and $C$ had finite rank was unnecessary. 
\begin{proposition}\label{prop:module_rank_for_short_exact_sequence_PID}
    Let $R$ be a PID and let
    $$
    0 \to A \xrightarrow{f} B \xrightarrow{g} C \to 0
    $$
    be a short exact sequence of $R$-modules. Then $\rk(A) + \rk(C) = \rk(B)$ in $\mathbb N_\infty$.
\end{proposition}
\begin{proof}
    The case when $A,B$ and $C$ have finite rank is already covered in Lemma \ref{lem:module_rank_for_short_exact_sequence_PID_finitely_generated}. Below we prove the formula for the cases when $A,B$ or $C$ have infinite rank.


    \begin{enumerate}[label=(\alph*)]
        \item Let $A$ have infinite rank. Then, for every $n \in \mathbb N$, there is an linearly independent subset $U \subseteq A$ with cardinality $n$.  
        Since $f$ is injective, $\{ f(u) | u \in U \}$ is still linearly independent. Since $n$ was arbitrary, this proves that $B$ has also infinite rank. Thus the equation
        $$
        \rk(A) + \rk(C) = \infty + \rk(C) = \infty = \rk(B)
        $$
        holds.
        \item Let $C$ have infinite rank. Then, for every $n \in \mathbb N$, there is an linearly independent subset $U \subseteq C$ with cardinality $n$. Since $g$ is surjective, there are preimages  $(b_u)_{u \in U}$ with $g(b_u)=u$. 
        They form a linearly independent subset of $B$ that has cardinality $n$.
        Since $n$ was arbitrary, $\rk(B) = \infty$, proving that 
        $$
        \rk (A) + \rk(C) = \rk(A) + \infty = \infty = \rk(B).
        $$
        \item Let $B$ have infinite rank and $A$ have finite rank. Then $\mathrm{Im}(f) \subset B$ has finite rank. Thus, by extending a maximal linearly independent subset of $\mathrm{Im}(f)$ to an arbitrary sized linearly independent subset of B, we obtain an arbitrarily sized linearly independent subset of $B/\mathrm{Im}(f)$. Due to the uniqueness of cokernels, $B / \mathrm{Im}(f) \cong C$ and thus $\rk(C)=\infty$. Now the equation
        $$
        \rk (A) + \rk(C) = \rk(A) + \infty = \infty = \rk(B)
        $$
        holds.
    \end{enumerate}
\end{proof}
    For any $R$-module $A$, the addition $+: A \times A \to A, (a,a') \mapsto a+a'$ and the negation $-: A \to A, a \mapsto -a$ are module-homomorphisms. Out of this we built the difference $f-g$ of module-homomorphisms. For $f: A \to B$ and $g: C \to B$ is defined as 
    $$
    f-g : A \times C \to B ~,~ (a,c) \mapsto f(a) - g(c).
    $$
    We will use this difference in Lemma \ref{lem:short_exact_sequence} and the proof of Proposition \ref{prop:rank_is_a_dimension}.
    Observe that $f-g$ is surjective if $f$ or $g$ is surjective.

\begin{lemma}\label{lem:short_exact_sequence}
    Let $A,B,C$ be $R$-modules. Let $A \times_B C$ denote the pullback arising as the limit of the span $A \xrightarrow{f} B \xleftarrow{g}C$. Then the following sequence is exact
    $$
    0 \to A \times_B C \xrightarrow{\langle \pi_0 , \pi_1\rangle} A \times C \xrightarrow{f-g} \mathrm{Im}(f-g) \to 0 
    $$
\end{lemma}
\begin{proof}
    The sequence is exact at $\mathrm{Im(f-g)}$ since any map is surjective onto its image. The sequence is exact at $A \times_B C$ since $\langle \pi_0, \pi_1 \rangle $ is injective.
    The kernel of $f-g$ is exactly the elements $(a,c)$ fulfilling $f(a) = g(c)$, thus it is equal to the image of $A \times_BC$ of $A \times C$.
\end{proof}

\begin{proposition}\label{prop:rank_is_a_dimension}
    Let $R$ be a PID.
    Then assigning the rank to $\rk(A)$ to an R-module $A$ 
    is an $(\mathbb N_\infty , +)$-valued dimension $\rk$ on $\Mod_R$.
\end{proposition}
\begin{proof}
    Let 
    \[\begin{tikzcd}
    	{A \times_B C} & A \\
    	C & B
    	\arrow[from=1-1, to=1-2]
    	\arrow[from=1-1, to=2-1]
    	\arrow["\lrcorner"{anchor=center, pos=0.125}, draw=none, from=1-1, to=2-2]
    	\arrow["r", from=1-2, to=2-2]
    	\arrow["f"', from=2-1, to=2-2]
    \end{tikzcd}\]
    be a pullback of $R$-modules of the form as the pullbacks in Definition \ref{def:monoid-valued-dimension}. Then Lemma \ref{lem:short_exact_sequence} shows that
    $$
    0 \to A \times_B C \xrightarrow{\langle \pi_0 , \pi_1\rangle} A \times C \xrightarrow{f-r} \mathrm{Im}(f-r) \to 0 
    $$
    is a short exact sequence. Since $r$ is surjective, $f-r$ is surjective i.e. $\mathrm{Im}(f-r)=B$, thus we obtain the short exact sequence
    $$
    0 \to A \times_B C \xrightarrow{\langle \pi_0 , \pi_1 \rangle} A \times C \xrightarrow{f-r} B \to 0.
    $$
    Therefore the rank formula in Proposition \ref{prop:module_rank_for_short_exact_sequence_PID} shows that $\rk(A \times C) = \rk (A \times_B C) + \rk(B)$. Due to the additivity of ranks under products this can be expanded into
    $$
    \rk (A) + \rk(C) = \rk(A \times_B C ) + \rk(B)
    $$
\end{proof}

Now we can apply Theorem \ref{thm:dimension_result_weak} and obtain a clear dimension formula for any tangent structure.
\begin{corollary}\label{cor:tangent_structures_on_Mod}
    Let $(T,p,0,+,\ell,c)$ be a tangent structure on $\Mod_R$. Then for any  $R$-module $A$, the equation
    $$
        \rk(T^2(A)) + 2 \cdot \rk(A) = 3 \cdot \rk(A)
    $$
    holds.
\end{corollary}
There are several special cases that follow from this
\begin{examples}\label{exs:vect_abGrp_as_examples_of_modules}
\begin{enumerate}
    \item Let $F$ be a field and let $\Vect_F$ denote the category of vector spaces over $F$ with $F$-linear maps as morphisms. Then the dimension is an $(\mathbb N_\infty, +)$ valued dimension and, for any tangent structure on $\Vect_F$,
    $$
        \dim(T^2(V)) + 2 \cdot \dim(V) = 3 \cdot \dim(V)
    $$
    holds for any vector space $V$.
    \item Let $\AbGrp$ be the category of abelian groups (i.e. modules over the integers). Then the rank of the abelian group is an $(\mathbb N_\infty, +)$ valued dimension and, for any tangent structure on $\AbGrp$,
    $$
        \rk(T^2(G)) + 2 \cdot \rk(G) = 3 \cdot \rk(G)
    $$
    holds for any abelian group $G$. Together with the cardinality from Theorem \ref{thm:cardinality_groups}, this means that the category of abelian groups has two distinct nontrivial $\mathbb N_\infty$-valued dimensions.
\end{enumerate}
\end{examples}
In the following two sections we will explain the consequences for finite dimensional vector spaces and Algebras over a PID.

\subsection{Module rank of an Algebra }\label{subsec:Alg_R}
In this section we utilize the free forgetful adjunction between modules and algebras to obtain a dimension of Algebras over a commutative ring $R$ with unit. According to \cite[Section IV.2]{maclane:71}, the forgetful functor $U : \Alg_R \to \Mod_R$ that sends a $R$-Algebra to its underlying $R$-module has a left-adjoint given by the tensor algebra. Thus we obtain a dimension on the category of $R$-Algebras. 

\begin{proposition}\label{prop:algebras_rank}
    Let $R$ be a commutative ring with unit and $\Alg_R$ be the category of $R$-algebras and Algebra-homomorphisms. 
    \begin{enumerate}[label = (\alph*)]
        \item The assignment $\rk$ that sends an algebra $A$ to the rank $\rk(U(A))$ of its underlying $R$-module is an $(\mathbb N_\infty,+)$-valued dimension on $\Alg_R$. 
        \item For any tangent structure $(\Alg_R  , T, p, 0,+,\ell, c)$ on the category $\Alg_R$ and for any object $A$,
            $$
            \rk(U(T^2(A))) + 2 \cdot  \rk(U(A)) = 3 \cdot \rk(U(T(A))).
            $$
    \end{enumerate}
\end{proposition}
\begin{proof}
    Part (a) follows from Corollary \ref{cor:adjunctions_full_subcats_and_dimensions}. Part (b) follows from Theorem \ref{thm:dimension_result_weak}.
\end{proof}
In particular algebras over $\mathbb Z$ are the same as rings. Thus we obtain an additional dimension on the category of rings.
\begin{corollary}\label{cor:rk_is_dimesnion_on_rings}
    Let $\rk_\mathbb Z$ denote the rank of an abelian group over the integers. 
    \begin{enumerate}[label = (\alph*)]
        \item The rank $\rk_\mathbb Z$ is an $\mathbb N_\infty$ valued dimension on $\Ring_n$, $\Ring_1$, $\Ring_u$, $\CRing_n$, $\CRing_1$ and $\CRing_u$.
        \item For any ring $R$ and any tangent structure $(T,p,0,+,\ell,c)$ on $\Ring_n$, $\Ring_1$, $\Ring_u$, $\CRing_n$, $\CRing_1$ or $\CRing_u$,
        $$
            \rk_\mathbb Z(T^2(A)) + 2 \cdot  \rk_\mathbb Z(A) = 3 \cdot \rk_\mathbb Z(T(A)).
        $$
    \end{enumerate}
\end{corollary}
\begin{proof}
    \begin{enumerate}[label = (\alph*)]
        \item Since $\Ring_n$ is the full subcategory of associative algebras in $\Alg_\mathbb Z$, it follows from Proposition \ref{prop:algebras_rank} and Corollary \ref{cor:adjunctions_full_subcats_and_dimensions} that $\rk_\mathbb Z$ is a dimension on $\Ring_n$. The categories $\Ring_1, \CRing_u$ and $\CRing_1$ are full subcategories of $\Ring_n$ and therefore Corollary \ref{cor:adjunctions_full_subcats_and_dimensions} shows that $\rk_\mathbb Z$ is a dimension on them.

        Pullbacks in $\Ring_u$ are also pullbacks in $\Ring_1$, because in either case commutative squares are pullbacks if and only if they are pullbacks of underlying sets. Thus any dimension on $\Ring_1$ is also a dimension on $\Ring_u$, in particular $\rk_\mathbb Z$. Since $\CRing_u$ is a full subcategory of $\Ring_u$, Corollary \ref{cor:adjunctions_full_subcats_and_dimensions} shows that $\rk_\mathbb Z$ is also a dimension on $\CRing_u$
        \item This now follows from \ref{thm:dimension_result_weak}.
    \end{enumerate}
\end{proof}
Thereby we now have two distinct non-trivial dimensions on $\Ring_1$ and $\Ring_u$, the characteristic (as shown in Section \ref{sec:rings}) and the rank.

If we consider the tangent structure of Example \ref{ex:tangent_on_ring}, we see that for any commmutative ring $R$ with unit,
$$
\rk_{\mathbb Z}(T(R)) = \rk_{\mathbb Z}(R[x]/(x^2)) = 2 \cdot \rk_\mathbb Z(R).
$$
Thus the tangent bundle functor doubles this dimension in this case like in the case for classical manifolds.

\subsection{Classification of Cartesian tangent structures on $\FinVect$}
The category of vector spaces over a field $F$ is exactly the category $\Mod_F$ of modules over $F$. The classical dimension $\Dim(V)$ of a vector space $V$ is exactly the rank $\rk(V)$. Due to Proposition \ref{prop:rank_is_a_dimension}, it is a $(\mathbb N_\infty, +)$-valued dimension.

All finite-dimensional vector spaces over a field $F$ are isomorphic to powers $F^n$ of the field $F$. Therefore the interaction between the tangent structure and the Cartesian structure dictates the entire tangent structure.
\begin{theorem}
    Let $F$ be a field and let $\FinVect_F$ be the full subcategory of $\Vect_F$ consisting of finite-dimensional vector spaces. Then every Cartesian tangent structure $(T,p,0,+,\ell ,c)$ on  $\FinVect_F$ fulfills either
    $$
    T(V) \cong V \quad  \forall V \in \mathrm{Obj}(\FinVect)
    $$
    or
    $$
    \qquad T(V) \cong V \times V\quad  \forall V \in \mathrm{Obj}(\FinVect).
    $$
\end{theorem}
\begin{proof}
    Every finite dimensional vector space is of the form $V = F^n$ where $n$ is the dimension of $V$. Thus every Cartesian tangent structure will fulfill $T(V) \cong T(F^n) \cong T(F)^n$. Thus $\Dim(T(V)) = \Dim(T(F)) \cdot \Dim(V)$ which means the vector space dimension is a strong tangent dimension valued in $\mathbb Z$. Now Theorem \ref{thm:dimension_argument_strong} proves that either $\Dim(T(V)) = 2 \cdot \Dim(V) ~\forall V$ or $\Dim(T(V)) = \Dim(V) ~\forall V$.
\end{proof}

In fact, both situations are known. The trivial tangent structure with $T(V) \cong V$ is explained in Example \ref{ex:trivial_tangent_structure}. The tangent structure that doubles the dimension is the tangent structure of Example \ref{ex:degenerate_module_tangent} for the case when $R$ is a field.

\section{Homotopy, excision and the opposite category of CW complexes}\label{sec:cw}
CW complexes are extensively used in topology to describe well-behaved topological spaces. Traditionally the dimension of a CW complex is the highest dimensional cell in the complex. However, this notion is not well-behaved under homotopy. We will show in this section that it is not a categorical dimension in the sense of Definition \ref{def:monoid-valued-dimension}.

Rational homology is a tool that describes a topological space $X$ in a homotopy-invariant way with a sequence of rational vector spaces $H_n(X,\mathbb Q)_{n \in \mathbb N}$. The dimension of these vector spaces are known as the Betti numbers of the space.

The $n$-th Betti number $B_n$ is often interpreted as ``counting $n$-dimensional holes" in a space. Since the Betti numbers are about the dimension of the complement of our space, it may seem intuitive to have them be a categorical dimension on the opposite category of CW complexes.

A \textbf{CW complex} is a topological space $X$, constructed by an inductive construction out of standard cells $D_k = \{\vec x \in \mathbb R^n | x^2\}$ for every $k \in \mathbb N$. The inductive construction can be found in Chapter 0 of \cite{HatcherAT}. The space obtained at the $k$-th stage of the inductive construciton is called $X^k$ the \textbf{$k$-skeleton} of X.

A continuous map $f: X \to Y$ between CW complexes $X = \cup_{k \in \mathbb N} X^k$ and $Y = \cup_{k \in \mathbb N} Y^k$ is called \textbf{cellular} if $f(X^k) \subseteq Y^k $ for all $k \in \mathbb N$.

\begin{definition}
    We define the category $\CW$ as the subcategory of $\Top$ consisting of 
    \begin{itemize}
        \item CW complexes of objects and
        \item cellular maps as morphisms.
    \end{itemize}
\end{definition}

\subsection{Not the classical dimension of CW complexes}
The classical dimension of a CW complex is considered to be the the largest $n$ of all $n$-cells used to construct it.
 
\begin{definition}
    Let $X = \cup_{ n \in \mathbb N} X^n$ be a CW complex.
    \begin{enumerate}[label = (\alph*)]
        \item The smallest number $n \in \mathbb N$ such that $X = X^n$ is the \textbf{classical dimension} of $X$.
        \item If there is no such $n$, the classical dimension is $\infty$.
    \end{enumerate}
\end{definition}
In particular for CW complexes $X$ and $X'$, the largest dimension cell of $X \times X'$ is the sum of the largest dimensions of cells in $X$ and $X'$, so desirable property \ref{item:cartesian_products_to_sums} (additive w.r.t. products) of Remark \ref{rem:wishlist_dimension} holds with respect to the addition.

While the classical dimension is the most common notion of dimension for CW complexes, we show it is not a dimension in the sense of Definition \ref{def:monoid-valued-dimension}.
We will construct a pullback of retractions in $\mathrm{CW}$ that serves as a counterexample for the dimension formula in Definition \ref{def:monoid-valued-dimension}. 

The idea behind this counterexample is to glue together a space $\mathrm{Bl}$ that has 1-dimensional and 3-dimensional components. In the pullback, the 3-dimensional components are prevented from combining with each other by having them sent to the opposite end of the interval $D^1$.
\begin{example}\label{ex:balloon_stick_CW}
We define the ``balloon shaped" CW complex $\mathrm{Bl}$ obtained by gluing the unit interval $D^1 = [0,1]$ with a 3-cell $D^3 = \{ (x,y,z) \in \mathbb R^3 | (x+1)^2 + y^2 + z^2 \leq  1 \}$ along the origin, obtaining a space that can explicitly be described as
$$
\{(x,0,0) \in \mathbb R^3 | 0 \leq x \leq 1\} \cup \{ (x,y,z) \in \mathbb R^3 | (x+1)^2 + y^2 + z^2 \leq  1 \}
$$
We define two maps $l,r: \mathrm{Bl} \to D^1 $, both contracting the 3-cell away. However, $l$ sends the end of the interval where the 3-cell was to the left and $r$ sends the end of the interval where the 3-cell was to the right. In explicit coordinates that means that
$$
l (x,y,z) = \left\lbrace 
\begin{matrix}
x & \text{if} &x\geq0 \\
0 & \text{if} &x<0
\end{matrix}
\right.
$$
$$
r(x,y,z) = \left\lbrace 
\begin{matrix}
1-x & \text{if} &x\geq0 \\
1 & \text{if} &x<0 .
\end{matrix} \right.
$$
It is easy to see that both $l$ and $r$ are retractions. 
As one can see directly from the Figure \ref{fig:baloon_pullback} the pullback
\[\begin{tikzcd}
	{\mathrm{Bl} \times_{D^1} \mathrm{Bl}} & {\mathrm{Bl}} \\
	{\mathrm{Bl}} & {D^1}
	\arrow[from=1-1, to=1-2]
	\arrow[from=1-1, to=2-1]
	\arrow["r", from=1-2, to=2-2]
	\arrow["l"', from=2-1, to=2-2]
\end{tikzcd}\]
will be an interval with a 3-cell glued to both ends. 

\begin{figure}
    \centering
    \includegraphics[width=0.5\linewidth]{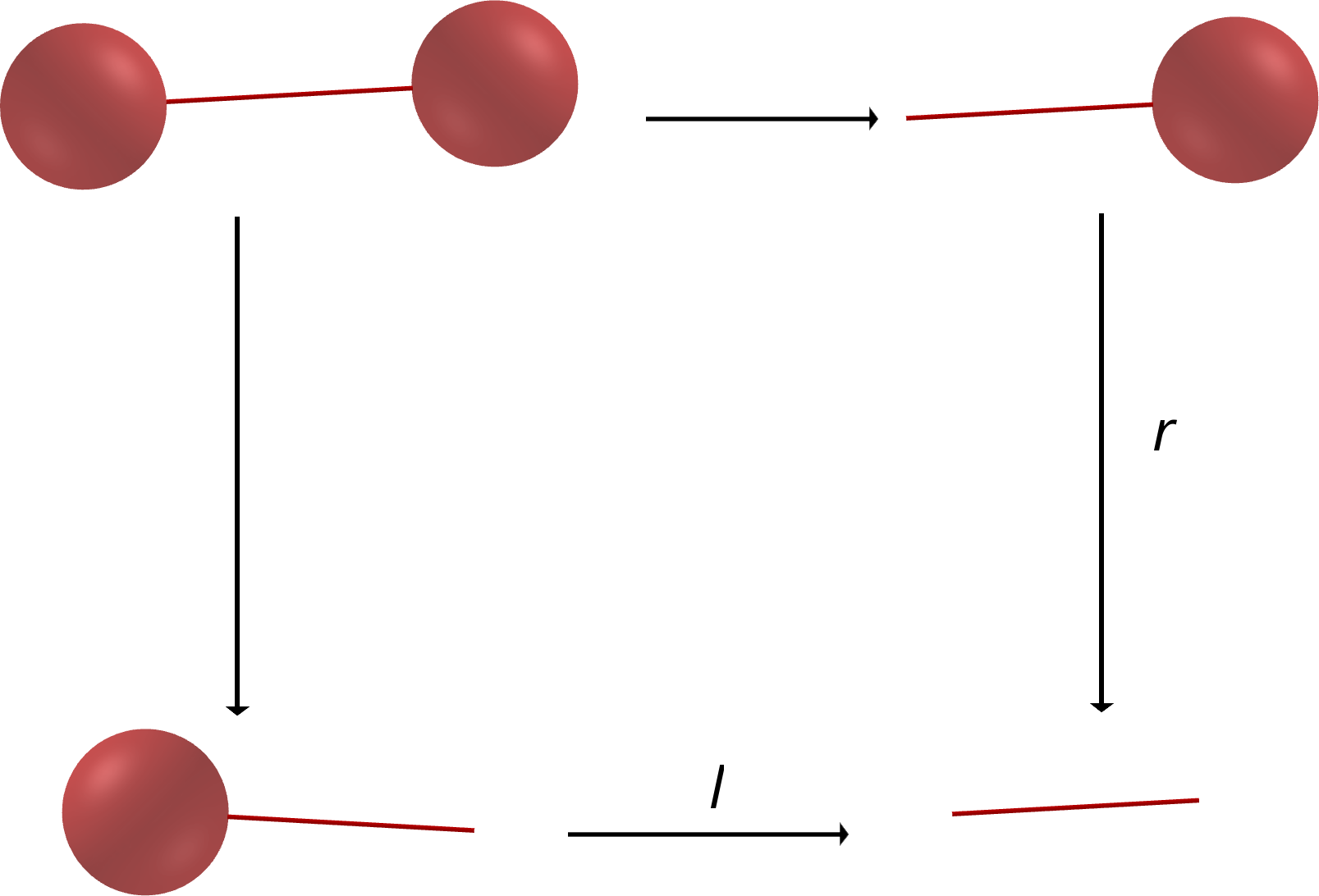}
    \caption{The pullback of the interval with a 3-cell at the left and the interval with a 3-cell at the right is an interval with a 3-cell on either side.}
    \label{fig:baloon_pullback}
\end{figure}

Since 
    $$
    \dim(\mathrm{Bl} \times_{D^1} \mathrm{Bl}) + \dim(D^1) = 3+1 \neq 3+3 = \dim(\mathrm{Bl}) + \dim(\mathrm{Bl}),
    $$
    the condition in Definition \ref{def:monoid-valued-dimension} does not hold.
    Thus the classical dimension of CW complexes is not $(\mathbb N, +)$-valued dimension on $\CW$ in the sense of Definition \ref{def:monoid-valued-dimension}.
\end{example}
In general, combinatrics from the gluing (i.e. colimits) produces counterexamples proving that the dimension of CW complexes is not a dimension in the sense of Definition \ref{def:monoid-valued-dimension}.

\subsection{Betti numbers and excisive functors}
In this section, we will show that the Betti numbers are a categorical dimension on the opposite category of CW complexes.

Intuitively, the $n$th Betti number counts the number of $n$-dimensional holes in a space $X$.
Following this intuition, we check Property \ref{item:cartesian_products_to_sums} (additive w.r.t. products) of Remark \ref{rem:wishlist_dimension}. The product in $\CW^\mathrm{op}$ is the disjoint union.
Since the disjoint union of two complexes with $n$ and $k$ holes has $n+k$ holes in total, desired Property \ref{item:cartesian_products_to_sums} (additive w.r.t. products) holds on $\CW^\mathrm{op}$ for the monoid $(\mathbb N_\infty, +)$.

We will show that the Betti numbers are indeed dimensions in the sense of Definition \ref{def:monoid-valued-dimension}. Thus the sequence of all Betti numbers will be a categorical dimension valued in infinite sequences. 


In \cite{eilenberg_steenrod}, Eilenberg and Steenrod define a generalized homology theory as an assignment sending a CW complex $X$ to a sequence, $H_n(X)$, of abelian groups that fulfills a list of 7 axioms (including functoriality). Singular homology is the motivating example of a generalized homology theory. 
It is a standard result, for example in \cite[Section 2.3]{HatcherAT}, that the Mayer-Vietoris sequence still exists for any generalized homology theory.  The Mayer-Vietoris sequence will be the main tool that we need in order to verify that the Betti number is a categorical dimension.  Thus our argument will hold for any generalized homology theory.

Let $H_n$ be denote a homology theory fulfilling the Eilenberg-Steenrod axioms.
Let $\CW$ be the category of CW complexes and cellular maps. Then the $n$-th Betti number $B_n(X)= \rk(H_n(X))$ is the rank of the $n$-th homology of a space $X$. It is easy to see that this equals the dimension $\dim(H_n(X;\mathbb Q))$ of the $n$-th rational homology of $X$.

\begin{proposition}\label{prop:CW_complexes_betti_dimension}
The $n$-th Betti number is an $\mathbb N_\infty$-valued dimension on $\mathrm{CW}^\mathrm{op}$.
\end{proposition}
\begin{proof}
A pullback along a retraction in $\mathrm{CW}^\mathrm{op}$ corresponds to a pushout along a section $s$ in CW. 
\[\begin{tikzcd}
	{C \sqcup_A B} & B \\
	C & A
	\arrow[from=1-2, to=1-1]
	\arrow[from=2-1, to=1-1]
	\arrow["s"', hook, from=2-2, to=1-2]
	\arrow["f", from=2-2, to=2-1]
\end{tikzcd}\]
Sections are exactly embeddings of retracts which are cellular embeddings of closed subspaces and therefore subcomplexes. Thus by Proposition 0.16 of \cite{HatcherAT} it has the homotopy extension property and is therefore a cofibration.


Section 10.7 of \cite{may1999concise} shows that the pushout along any cofibration is homotopy-equivalent to the double mapping cylinder.
$$
C \sqcup_A B \simeq {_CM_{B}} = ( C \sqcup (A \times I) \sqcup B ) / \sim , \text{ where }(a,0) \sim f(a)\text{ and }(a,1) \sim s(a).
$$

This double mapping cylinder is the union of the two open sets $X = C \sqcup (A \times [0,1)) / \sim$ and $Y = (A \times (0,1]) \sqcup B / \sim$. Their intersection $X \cap Y = A \times (0,1)$ is homotopy-equivalent to $A$. 
Thus we obtain the Mayer Vietoris sequence
$$
...\to H_{n+1}(_CM_B) \xrightarrow{~0~}  
H_n(A) \xrightarrow{\langle H_n(s),H_n(f) \rangle} 
H_n(B) \oplus H_n(C) \longrightarrow H_n(_CM_B) \xrightarrow{~0~} H_{n-1}(A) \to ...
$$
where the maps $H_{n+1}(_CM_B) \to H_n(A)$ are zero since $H_n(s)$ is a section and therefore the kernel of $(\langle H_n(s),H_n(f) \rangle)$ is zero.
Therefore
$$
B_n(C \sqcup_A B) + B_n(A) = B_n(C) + B_n(B).
$$
and thus $B_n$ is an $\mathbb N_\infty$-valued dimension on CW$^\mathrm{op}$.
\end{proof}

Instead of considering the Betti numbers as an infinite sequence of dimensions, we can also consider them as a dimension valued in infinite sequences of numbers: Let $\mathbb N^\infty_\infty$ be the set of infinite sequences (or infinite vectors) in $\mathbb N_\infty$. It has a commutative monoid structure $+$, given by component-wise addition.
\begin{corollary}
    The sequence of Betti numbers 
    $$
    (B_n)_{n \in \mathbb N} : X \mapsto (B_n(X))_{n \in \mathbb N}
    $$
    is an $(\mathbb N_\infty^\infty , +)$-valued dimension on $CW^\mathrm{op}$.
\end{corollary}

\begin{corollary}
    For every tangent structure $T$ on $CW^{op}$, the Betti numbers $(B_n)_{n \in \mathbb N}$ satisfy 
    $$
    B_n(T^2(X)) + 2 \cdot B_n(X) = 3 \cdot B_n(T(X)).
    $$
\end{corollary}

We will now explore a counterexample of a pushout of Hausdorff spaces along a section which is not homotopy-equivalent to the double mapping cylinder. This serves to show that our proof of Proposition \ref{prop:CW_complexes_betti_dimension} does not work in the more general setting of Hausdorff spaces and continuous maps.
\begin{example}\label{ex:hausdorff_cylinder_not_pushout}
 Consider the Hausdorff spaces
\begin{align*}
    A' =& \left\lbrace x \in \mathbb R | x=\frac{1}{n} \text{ for some } n \in \mathbb N_{>0}\right\rbrace, 
    \\
    B' =& \left\lbrace (x,y,z) \in \mathbb R^3 | (x,y,z)=\left(\frac{1}{n} , \frac{1}{n} \sin(\varphi) , \frac{1}{n} \cos(\varphi) \right) \text{ for some } n \in \mathbb N_{>0} ~,~ \varphi \in [0,2\pi]\right\rbrace,
    \\
    A =& \{ 0 \} \cup A' ,  \text{ and }
    \\
    B =& \{ (0,0,0) \} \cup B'
\end{align*}
and the pushout diagram
\[\begin{tikzcd}
	{*} & A \\
	B & P
	\arrow["{i_0}", from=1-1, to=1-2]
	\arrow["{i_{(0,0,0)}}"', from=1-1, to=2-1]
	\arrow[from=1-2, to=2-2]
	\arrow[from=2-1, to=2-2]
	\arrow["\lrcorner"{anchor=center, pos=0.125, rotate=180}, draw=none, from=2-2, to=1-1]
\end{tikzcd}\]
where the strict pushout is 
$$
P = \{(-x,0,0)|x \in A\} \cup B \subset \mathbb R^3.
$$
As shown in Appendix A and visualized in Figure \ref{fig:M_isnt_the_pushout}, the double mapping cylinder is not homotopy-equivalent to this strict pushout.
\end{example}
\begin{figure}
    \centering
    \includegraphics[width=0.9\linewidth]{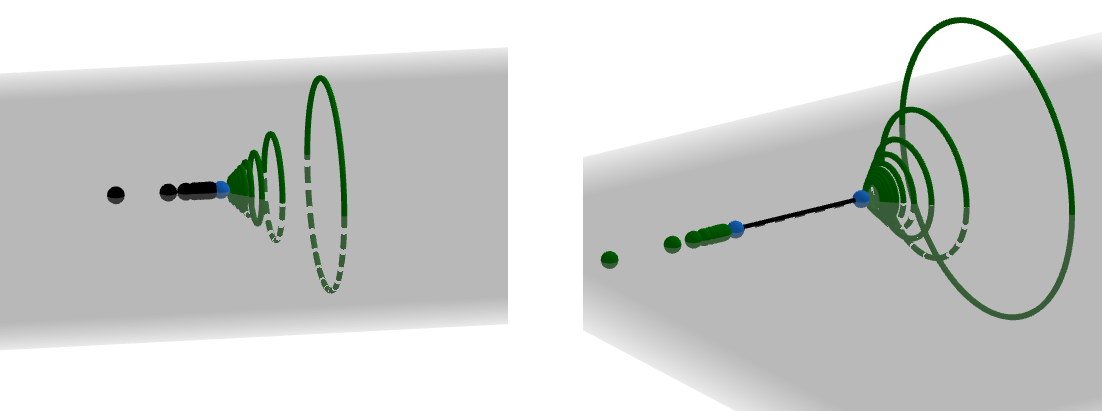}
    \caption{On the left is the pushout of $A$ and $B$ from Example \ref{ex:hausdorff_cylinder_not_pushout}, on the right the double mapping cylinder $_AM_B$. They are not homotopy-equivalent. The intuitive reason is that every open neighbourhood around the blue point on the left contains infinitely many dots and circles, whereas on the right there is no such point.}
    \label{fig:M_isnt_the_pushout}
\end{figure}

The most important property we used in the proof of Proposition \ref{prop:CW_complexes_betti_dimension} was that a generalized homology theory sends the pushout square
\[\begin{tikzcd}
	{C \sqcup_A B} & B \\
	C & A
	\arrow[from=1-2, to=1-1]
	\arrow[from=2-1, to=1-1]
	\arrow["s"', hook, from=2-2, to=1-2]
	\arrow["f", from=2-2, to=2-1]
\end{tikzcd}\]
to a pullback square of modules and that the rank of a module is a dimension. In homotopy theory, homotopy functors that send homotopy pushout squares to homotopy pullback squares are generally known as \textbf{1-excisive functors}. Below we generalize Proposition \ref{prop:CW_complexes_betti_dimension} for 1-excisive functors.

In the proposition below, we assume that homotopy limits/colimits are the same as strict limits/colimits.  This allows us to compare the hypotheses in Definition \ref{def:monoid-valued-dimension} (which uses strict limits/colimits) to the properties of excisive functors (which are defined by their effect on homotopy limits/colimits).
\begin{proposition}\label{prop:excisive_funcotrs_preserve_dimension}
    Let $\mathbb X$ be a model category in which (strict) pushouts of sections along sections and (strict) pushouts of sections along retractions are homotopy pushouts.
    Let $\mathbb Y$ be a complete and cocomplete category. We will consider it as a model category with trivial model structure.    
    Let $F: \mathbb X  \to \mathbb Y$ be an 1-excisive functor.    
    If $\dim$ is a categorical dimension on $\mathbb Y$, then the assignment $\dim \circ F$ that sends and object $X$ of $\mathbb X$ to $\dim(F(\mathbb X))$ is a categorical dimension on $\mathbb X^\mathrm{op}$.
\end{proposition}
\begin{proof}
    The pullbacks of Definition \ref{def:monoid-valued-dimension} in $\mathbb X^\mathrm{op}$ are pushouts in $\mathbb X$ and by assumption they are also homotopy pushouts. Since $F$ is excisive, it sends them to pullbacks in $\mathbb Y$. The dimension formula for $\dim$ now implies the dimension formula for $\dim \circ F$.
\end{proof}

\begin{remark}
    It may be possible to define a \textbf{homotopy categorical dimesnion} using homotopy limits/colimits in place of the strict limits/colimits of Definition \ref{def:monoid-valued-dimension}. 
    This homotopy dimension would be preserved (analogously to Proposition \ref{prop:excisive_funcotrs_preserve_dimension}) by any excisive functors between any model categories, without the requirement that homotopy limits/colimits and strict limits/colimits coincide.

    Given some higher tangent structure, like in \cite{BauerBurkeChing}, where the pullbacks of the tangent structure are replaced by homotopy pullbacks, this homotopy dimension should give rise to results analogous to Theorems \ref{thm:dimension_result_weak}, \ref{thm:dimension_argument_strong}, \ref{thm:dimension_result_weak} and \ref{thm:dimension_argument_strong}.

\end{remark}

\begin{remark}\label{rem:betti_number_inequalities}
    Desired Property \ref{item:embeddings_inequalities} (embeddings to inequalities) of Remark \ref{rem:wishlist_dimension} does not hold for the Betti numbers as dimensions on $\CW^\mathrm{op}$. There are monomorphisms $A \hookrightarrow A'$ in  $\CW^\mathrm{op}$ (i.e. epimorphisms $A' \rightarrow A$ in $\CW$) such that $B_n(A) > B_n(A')$. One example of this is the map 
    $$
    [0,1]=I \to S^1 = \{e^{ix}\} \subset \mathbb C~,~x \mapsto e^{3\pi x}.
    $$
    Since it winds more than once around the circle it is surjective and thus an epimorphism and $B_1(I)=0$ and $B_1(S^1)=1$.

    However, if we only consider sections $A \hookrightarrow A'$ in $\CW^\mathrm{op}$ (i.e. retractions in \CW), Property \ref{item:embeddings_inequalities} (embeddings to inequalities) follows from functoriality of homology. In the future, we hope this insight can help us refine the definition of dimension into a functorial and universal construction.  
\end{remark}

\section{Concluding remarks}

The goal of this notion of dimension is not to generalize many classical notions of dimensions (though it generalizes at least some). The goal is to give an extra tool when searching for tangent structures on a given category. In particular, if one has a categorical dimension in mind and wishes for it to be compatible with a tangent structure (so that the categorical dimension becomes a strong dimension with respect to the tangent structure), then one only needs to consider tangent bundle endofunctors which leave the dimension invariant or for which the tangent bundle endofunctor doubles the dimensions.  This can exclude many cases, thereby offering another tool in constructing tangent structures.

\subsubsection*{Ideas for further examples}
Across the literature there are many more notions of ``sizes'' of objects. For example the magnitude for finite metric spaces and graphs, first introduced by Andrew Solow and Stephen Polasky in \cite{biodiversity} and further elaborated by Tom Leinster in \cite{Leinster_magnitude} has some promising properties that suggests it may be a dimension. 

Additionally, in section \ref{sec:cw} we saw that from any generalized homology theory one obtains an infinite family of dimensions. We suspect that similar results might holds for homologies of graphs and links. Finding dimensions in such traditionally non-geometric areas will make it possible to find tangent structures in them and obtain a geometric perspective to areas that are not traditionally considered as geometry.

\subsubsection*{Relations to the tangent category literature}
In \cite{Leung2017}, Leung classifies tangent structures on a category $\mathbb X$ in terms of functors from the category $\Weil$ of Weil-Algebras into the endofunctors of $\mathbb X$. One can show that the rank of the underlying $\mathbb N$-module is a categorical dimension on $\Weil$.
Thus, in the future, we will investigate if there is a way to construct a nontrivial dimension on a given tangent category using the dimension on the category $\Weil$. If we prove such a result for every tangent category, this will provide another tool to prove that a certain category has only the trivial tangent structure by proving that there is no nontrivial categorical dimension on it.

In \cite{coalgebras_of_differential_categories}, Cockett, Lemay and Rory Lucyshyn-Wright proved that coEilenberg-Moore categories of differential categories are tangent categories. Therefore we are curious if there is some notion of dimension that describes differential categories, similar to the way how the dimensions of Definition \ref{def:monoid-valued-dimension} describe tangent categories.

While we defined dimensions using retractions, in differential geometry any pullback along a submersion fulfills desired Property \ref{item:pullbacks} (additive w.r.t. pullbacks). In \cite{lanfranchi2025tangentdisplaymaps}, Cruttwell and Lanfranchi prove that display maps are a generalization of submersions in the setting of tangent categories.
Thus, in order to gain a better understanding of dimensions, it seems worthwhile to investigate the behaviour of categorical dimensions with respect to pullbacks along display maps.

\subsubsection*{Universal dimension}
The notion of dimension in Definitions \ref{def:monoid-valued-dimension} and \ref{def:tangent_dimension} is the dimension one needs for the purpose of understanding the universality of the vertical lift. 

We hope that eventually we will discover a closely related notion of dimension that is universal with respect to some list of desired properties similar to Remark \ref{rem:wishlist_dimension}, but this list has not been determined here. Definition \ref{def:monoid-valued-dimension} does not fully characterize a dimension, as showcased by the fact that there are multiple nontrivial dimensions on given categories. For example the cardinality and the rank are both nontrivial dimensions on the category of abelian groups.

In we listed some desirable properties of dimension, some of which our notion of dimension fulfills and some of which it does not fulfill. Our definition of a dimension is not universal and in particular there are multiple dimensions on a given category.

Desired Property \ref{item:embeddings_inequalities} (embeddings to inequalities) of Remark \ref{rem:wishlist_dimension} states that embeddings of sub-manifolds are sent to inequalities of dimensions: If $M' \hookrightarrow M$ is a submanifold of $M$, then $\dim(M') \leq \dim(M)$.

While Remark \ref{rem:betti_number_inequalities} shows that for the Betti numbers not all monomorphisms give rise to inequalities, all sections are sent to inequalities by all examples we gave.

We hope that this observation can motivate a future generalization of dimension which is not just an assignment of objects, but actually a functor.

\bibliography{sample}

@misc{geoff_peripatetic,
    author = {G. Vooys},
    year = {February 20th 2025},
    note = {personal communication}
}

@book{brendon_topo_and_geo,
    author = {G. E. Brendon},
    title = {Topology and Geometry},
    publisher = {Springer New York, NY},
    series ={ Graduate Texts in Mathematics },
    year = {1993}
}

@book{Leee_intro_smooth_manifolds,
    author = {J. M. Lee},
    series ={ Graduate Texts in Mathematics },
    title = {Introduction to Smooth Manifolds},
    publisher = {Springer New York, NY},
    year = {2012},
    edition = {2}
}

@article{eilenberg_steenrod,
    author = {S. Eilenberg and N.E. Steenrod},
    title = {Axiomatic Approach to Homology Theory},
    journal = {Proc. Natl. Acad. Sci. U.S.A.},
    volume = {31 (4)},
    pages = {117-120},
    year = {1945}
}

@book{Vakil_rising_sea,
    author = {R. Vakil},
    title = {The Rising Sea: Foundations of Algebraic Geometry},
    publisher = {Princeton University Press},
    year = {2025}
}

@book{Gallian_contemporary_abstract_algebra,
    author = {J. A. Gallian},
    title = {Contemporary Abstract Algebra},
    publisher = {Belmont, Calif. : Brooks/Cole, Cengage Learning},
    year = {2010}
}

@book{Kaplansky_infinite_abelian_groups,
author = {Kaplansky, I.},
series = {Dover Books on Mathematics},
publisher = {Dover Publications},
title = {Infinite abelian groups.},
year = {2018}
}

@InProceedings{coalgebras_of_differential_categories,
  author =	{Cockett, R. and Lemay, J.-S. P. and Lucyshyn-Wright, R. B. B.},
  title =	{{Tangent Categories from the Coalgebras of Differential Categories}},
  booktitle =	{28th EACSL Annual Conference on Computer Science Logic (CSL 2020)},
  pages =	{17:1--17:17},
  series =	{Leibniz International Proceedings in Informatics (LIPIcs)},
  ISBN =	{978-3-95977-132-0},
  ISSN =	{1868-8969},
  year =	{2020},
  volume =	{152},
  editor =	{Fern\'{a}ndez, Maribel and Muscholl, Anca},
  publisher =	{Schloss Dagstuhl -- Leibniz-Zentrum f{\"u}r Informatik},
  address =	{Dagstuhl, Germany},
  URL =		{https://drops.dagstuhl.de/entities/document/10.4230/LIPIcs.CSL.2020.17},
  URN =		{urn:nbn:de:0030-drops-116607},
  doi =		{10.4230/LIPIcs.CSL.2020.17}
}

@book{Aluffi,
    series={Graduate studies in mathematics},
	title = { Algebra: Chapter 0 },
	isbn = {978-1-4704-6571-1},
    publisher={American Mathematical Society},
	author = {Aluffi, P.},
	year = {2009}
}

@book{Eisenbud:CommutativeAlgebrawithaViewTowardAlgebraicGeometry,
title = { Commutative Algebra with a View Toward Algebraic Geometry},
author = {Eisenbud, D.},
series = {Graduate Texts in Mathematics},
publisher = {Springer New York},
isbn = {978-0-387-94268-1},
year={1995}
}

@article{Leinster_magnitude,
    author = {T. Leinster},
    title = {The magnitude of metric spaces},
    journal = {Doc. Math.},
    year = {2013},
    volume ={18},
    pages = {857–905}
}

@article{biodiversity,
    author = {A. R. Solow and S. Polasky},
    title = {Measuring biological diversity},
    journal = {Environmental
and Ecological Statistics},
    volume = {1},
    pages ={95-107},
    year = {1994}
}

@book{spivak_book, place={Cambridge}, series={London Mathematical Society Lecture Note Series}, title={Polynomial Functors: A Mathematical Theory of Interaction}, publisher={Cambridge University Press}, author={Niu, N. and Spivak, D. I.}, year={2025}, collection={London Mathematical Society Lecture Note Series}}

@misc{McAdam_arxiv,
      title={Vector bundles and differential bundles in the category of smooth manifolds}, 
      author={B. MacAdam},
      publisher={arXiv:2007.11708},
      year={2020},
      eprint={2007.11708},
      archivePrefix={arXiv},
      primaryClass={math.CT},
      url={https://arxiv.org/abs/2007.11708}, 
}

@misc{ikonicoff2025abelianizationtangentcategories,
      title={From Abelianization to Tangent Categories}, 
      author={S. Ikonicoff and J.-S. P. Lemay and T. Van der Linden},
      year={2025},
      publisher={arXiv:2510.12324},
      archivePrefix={arXiv},
      primaryClass={math.CT},
      url={https://arxiv.org/abs/2510.12324}
}

@misc{lanfranchi2025tangentdisplaymaps,
      title={Pullbacks in tangent categories and tangent display maps}, 
      author={G. Cruttwell and M. Lanfranchi},
      year={2025},
      publisher={arXiv:2502.20699},
      eprint={2502.20699},
      archivePrefix={arXiv},
      primaryClass={math.CT}
}

@book{HatcherAT,
      author        = "Hatcher, A.",
      title         = "{Algebraic topology}",
      publisher     = "Cambridge Univ. Press",
      address       = "Cambridge",
      year          = "2000",
      url           = "https://cds.cern.ch/record/478079",
}

@book{maclane:71,
  address = {New York},
  author = {MacLane, S.},
  mrclass = {18-02},
  mrnumber = {MR0354798 (50 \#7275)},
  mrreviewer = {H.-B. Brinkmann},
  note = {Graduate Texts in Mathematics, Vol. 5},
  pages = {ix+262},
  publisher = {Springer-Verlag},
  title = {Categories for the Working Mathematician},
  year = 1971
}

@article{Cockett2014DifferentialST,
  title="Differential Structure, Tangent Structure, and SDG",
  author="J. R. B. Cockett and G. S. H. Cruttwell",
  journal="Applied Categorical Structures",
  year="2014",
  volume="22",
  pages="331-417"
}

@article{rosicky,
     author = {Rosick\'y, J.},
     title = {Abstract tangent functors},
     journal = {Diagrammes},
     note = {talk:3},
     publisher = {Universit\'e Paris 7, Unit\'e d'enseignement et de recherche de math\'ematiques},
     volume = {12},
     year = {1984},
     zbl = {0561.18008},
     mrnumber = {800500},
     language = {en},
     url = {http://www.numdam.org/item/DIA_1984__12__A3_0/}
}

@article{cockett2016diffbundles,
  doi = {10.48550/ARXIV.1606.08379},
  url = {https://arxiv.org/abs/1606.08379},
  author = {Cockett, J. R. B. and Cruttwell, G. S. H.},
  title = {Differential bundles and fibrations for tangent categories},
  year = {2016}
}

@book{may1999concise,
  title={A Concise Course in Algebraic Topology},
  author={May, J.P.},
  isbn={9780226511832},
  lccn={99296133},
  series={Chicago Lectures in Mathematics},
  url={https://books.google.ca/books?id=g8SG03R1bpgC},
  year={1999},
  publisher={University of Chicago Press}
}

@article{Leung2017,
    author = {P. Leung},
    title = {Classifying tangent structures using {W}eil algebras},
    journal = {Theory Appl. Categ.},
    volume = {32},
    pages = {286-337},
    year = {2017}
}

@misc{BauerBurkeChing, 
  doi = {10.48550/ARXIV.2101.07819},
  url = {https://arxiv.org/abs/2101.07819},
  author = {Bauer, K. and Burke, M. and Ching, M.},
  title = {Tangent infinity-categories and {G}oodwillie calculus},
  publisher = {arXiv},
  year = {2021},
  copyright = {Creative Commons Attribution Non Commercial No Derivatives 4.0 International}
}
\clearpage
\appendix
\section{When the pushout is not the double mapping cylinder}
The purpose of this appendix is to prove the statements made in Example \ref{ex:hausdorff_cylinder_not_pushout}. Concretely, this appendix is a proof of Proposition \ref{prop:hausdorff_cylinder_not_pushout}.
The goal of this is to show that, while they form a dimension on the category of CW complexes and cellular maps, our proof does not work for the Betti numbers as a dimension on the category of Hausdorff spaces and continuous maps.

First we will recall some definitions used in the statements.

\begin{definition}\label{def:homotopy_equivalence_components_doublemappingcylinder}
Let $X, Y$ be topological spaces and $f,f': X \to Y$ be continuous maps. Let $I=[0,1]$ denote the closed unit interval.
\begin{enumerate}[label = (\alph*)]
    \item The maps $f$ and $f'$ are \textbf{homotopic}, if there is a continous map $h:I \times X \to Y$ such that $h(0,x) = f(x)$ and $h(1,x)=f'(x)$ for all $x \in X$. Such a map $h$ is called a \textbf{homotopy}.
    \item A \textbf{homotopy equivalence} between $X$ and $Y$ consists of continuous maps $g: X \to Y$ and $g': Y \to X$ and homotopies $h$ between $g \circ g'$ and the identity $1_Y$ and $h'$ between $g' \circ g$ and the identity $1_X$.
    \item Two points $x,y \in X$ are \textbf{path-connected} if there is a path between them, i.e. a continuous map $\gamma: I \to X$ such that $\gamma(0) = x$ and $\gamma(1) = y$. A \textbf{path-component} of $X$ is a maximal path-connected subset.
    \item Given a span $X \xleftarrow{f} Y \xrightarrow{g} Z$ of topological spaces, the double mapping cylinder $_XM_Z$
$$    
_XM_Z = ( X \sqcup (Y \times I) \sqcup Z ) / \sim , \text{ where }(y,0) \sim f(y)\text{ and }(y,1) \sim g(y), ~ \forall y \in Y.$$
\end{enumerate}
\end{definition}

The following proposition is the central result of this appendix.
\begin{proposition}\label{prop:hausdorff_cylinder_not_pushout}
    For the Hausdorff spaces 
\begin{align*}
    A' =& \left\lbrace x \in \mathbb R | x=\frac{1}{n} \text{ for some } n \in \mathbb N_{>0}\right\rbrace, 
    \\
    B' =& \left\lbrace (x,y,z) \in \mathbb R^3 | (x,y,z)=\left(\frac{1}{n} , \frac{1}{n} \sin(\varphi) , \frac{1}{n} \cos(\varphi) \right) \text{ for some } n \in \mathbb N_{>0} ~,~ \varphi \in [0,2\pi]\right\rbrace,
    \\
    A =& \{ 0 \} \cup A' ,  \text{ and }
    \\
    B =& \{ (0,0,0) \} \cup B'
\end{align*}
and the pushout diagram
\[\begin{tikzcd}
	{*} & A \\
	B & P
	\arrow["{i_0}", from=1-1, to=1-2]
	\arrow["{i_{(0,0,0)}}"', from=1-1, to=2-1]
	\arrow[from=1-2, to=2-2]
	\arrow[from=2-1, to=2-2]
	\arrow["\lrcorner"{anchor=center, pos=0.125, rotate=180}, draw=none, from=2-2, to=1-1]
\end{tikzcd}\]
there is no homotopy equivalence between $P$ and the double mapping cylinder $_AM_B$
\end{proposition}
Before we prove this result, we will state the intuition that homotopy-equivalences preserve path components.
\begin{lemma}\label{lem:homotopies_preserve_path_components}
    Let $A,B$ be topological spaces and let there be a homotopy equivalence between $A$ and $B$ consisting of $f: A \to B$, $g: B \to A$, $h: A \times I \to A$ and $k: B \times I \to B$. 
    \begin{enumerate}[label = (\roman*)]
        \item If $a,a' \in A$ can be connected by a path $I \to A$, $f(a)$ and $f(a')$ can be connected by a path $I \to B$.
        \item If $a,a'$ can not be connected by a path $I \to A$, $f(a)$ and $f(a')$ can not be connected by a path $I \to B$.
        \item Let $X \subset A$ be a path component of $A$ and $Y \subset B$ a path component of $B$ and let for some $x \in X$ the image $f(x)$ be in $Y$. Then the restrictions $(f|_X, g|_Y,h|_{X \times I}, k|_{Y \times I})$ form a homotopy equivalence between $X$ and $Y$.
    \end{enumerate}
\end{lemma}
\begin{proof}
    Part (i) is immediate using the the post-composition of the path with $f$.

    For part (ii), suppose that $f(a)$ and $f(a')$ are connected by a path $\gamma: I \to B$. Then $g \circ \gamma: I \to A$ is a path in $A$ connecting $g(f(a))$ with $g(f(a'))$. The homotopy $h$ between $g \circ f$ and $1_A$ now provides paths $h(a,-): I \to A$ and $h(a',-)$ connecting $a$ with $g(f(a))$ and $a'$ with $g(f(a'))$. Composing these paths gives a path from $a$ to $a'$ proving part (ii).

    For part (iii), we need to show that the image of the restrictions is within $X$ and $Y$ respectively. 
    
    Given a $x'$ within the same path component $X$ as $x$, part (i) shows that $f(x')$ will be within the same path component as $f(x)$ which is $Y$. For any $t \in I$, $h(x',t)$ has a path connecting it to $h(x',1)=x'$, so it is within the same path component $X$ as $x'$.

    Since $h(x,-): I \to A$ is a path between $g(f(x))$ and $x$, the element $f(x) \in Y$ is sent to the connected component of $x$. Thus we can use the same argument as in the previous paragraph to obtain that the images of $g|_Y$ and $k|_{Y \times I}$ land in $X$ and $Y$ respectively.

    Now since we showed that the restrictions have the types $f|_X: X \to Y$, $g|_Y: Y \to X$, $h|_{X \times I}: X \times I \to X$ and $ k|_{Y \times I}: Y \times I \to Y$, they form a homotopy equivalence between $X$ and $Y$, since $(f,g,h,k)$ formed a homotopy equivalence between $A$ and $B$.
\end{proof}
Equipped with this Lemma we now can prove Proposition \ref{prop:hausdorff_cylinder_not_pushout}.
\begin{proof}[Proof of Proposition \ref{prop:hausdorff_cylinder_not_pushout}.]
    We use the explicit descriptions
    $$
    P = \left\lbrace(-\frac{1}{n},0,0)|n \in \mathbb N_{>0} \right\rbrace \cup \{(0,0,0)\} \cup \left\lbrace (\frac{1}{n},\frac{1}{n}\sin(\varphi), \frac{1}{n}\cos(\varphi)) | n \in \mathbb N_{>0} , \varphi \in [0,2 \pi] \right\rbrace \subset \mathbb R^3 \text{ and}
    $$
    $$
    _AM_B \cong \left\lbrace(-\frac{1}{n},0,0)|n \in \mathbb N_{>0} \right\rbrace \cup [0,1] \times \{(0,0)\} \cup \left\lbrace (1+\frac{1}{n},\frac{1}{n}\sin(\varphi), \frac{1}{n}\cos(\varphi)) | n \in \mathbb N_{>0} , \varphi \in [0,2 \pi] \right\rbrace \subset \mathbb R^3 .
    $$
    These explicit descriptions are visualized in Figure \ref{fig:M_isnt_the_pushout}.
    This proof proceeds by contradiction. Suppose there is a homotopy equivalence $(f,g,h,k)$ between $P$ and $_A M_B$. Then for $\vec 0 = (0,0,0)$ we consider the point $p_0 =f(\vec 0) \in \,_AM_B$ and a small open neighbourhood $U_0$ around it. Since $f$ is continuous $f^{-1}(U_0)$ is an open neighbourhood around $\vec 0$. Thus for any open neighbourhood $U_0$ of $p_0$, $f^{-1}(U_0)$ contains infinitely many points of the form $(-1/n,0,0)$ and infinitely many circles of the form $\{(1/n,1/n \sin(\varphi), 1/n \cos(\varphi))| \varphi \in [0,2\pi]\}$. We will now distinguish different cases where $p_0$ could be, showing that each of them leads to a contradiction
    \begin{itemize}
        \item If $p_0 = (-1/n,0,0) \in {}_AM_B$ for some $n \in \mathbb N_{>0}$, $U_0$ can be chosen to be just the set $U_0=\{p_0\}$ which is open in subset topology as an open ball of radius $1/n-1/(n+1)$. Since $f^{-1}(U_0)$ contains infinity many path components, $f$ sends infinity many path components to one path component. By Lemma \ref{lem:homotopies_preserve_path_components}(ii) this contradicts the assumption that $(f,g,h,k)$ was a homotopy-equivalence.
        \item If $p_0 = (1+1/n,1/n \sin(\varphi),1/n \cos(\varphi)) \in {}_AM_B$ for some $n \in \mathbb N_{>0}$ and some $\varphi \in [0,2\pi]$, $U_0$ can be chosen to be the open set $\{(1+1/n,1/n \sin(\varphi), 1/n \cos(\varphi))| \varphi \in [0,2\pi]\}$ (open in subspace topology). Since $f^{-1}(U_0)$ contains infinity many path components, $f$ sends infinity many path components to one path component. By Lemma \ref{lem:homotopies_preserve_path_components}(ii) this contradicts the assumption that $(f,g,h,k)$ was a homotopy-equivalence.
        \item If $p_0 = (x,0,0)  \in {}_AM_B$ for some $x \in (0,1)$, $U_0$ can be chosen to be $(0,1) \times \{(0,0)\}$ (open in subspace topology). Since $f^{-1}(U_0)$ contains infinity many path components, $f$ sends infinity many path components to one path component. By Lemma \ref{lem:homotopies_preserve_path_components}(ii) this contradicts the assumption that $(f,g,h,k)$ was a homotopy-equivalence.
        \item If $p_0=(0,0,0) \in {}_AM_B$, $U_0$ can be chosen to have a first coordinate below $1/2$. Since $f^{-1}(U_0)$ contains a circle of the form $\{(1/n,1/n \sin(\varphi), 1/n \cos(\varphi))| \varphi \in [0,2\pi]\}$ and these circles are in a different path component than $\vec 0$, Lemma \ref{lem:homotopies_preserve_path_components}(ii) implies that this circle is sent to a point of the form $(-1/m,0,0)$. Due to Lemma \ref{lem:homotopies_preserve_path_components}(i) and the fact that all these point are path-disconnected, the entire circle is sent to one point. By Lemma \ref{lem:homotopies_preserve_path_components}(iii) this implies that $(f,g,h,k)$ restricts to a homotopy equivalence between a circle and a point. It is a classical result that such a homotopy equivalence does not exist (proven for example by observing that their fundamental groups differ) leading to the required contradiction.
        \item If $p_0 = (1,0,0) \in {}_AM_B$, $U_0$ can be chosen to have a first coordinate above $1/2$. Since $f^{-1}(U_0)$ contains a isolated point of the form $(-1/n,0,0)$, Lemma \ref{lem:homotopies_preserve_path_components}(ii) implies that this point is sent to one of the circles of the form $\{(1+1/m,1/m \sin(\varphi), 1/m \cos(\varphi))| \varphi \in [0,2\pi]\}$. By Lemma \ref{lem:homotopies_preserve_path_components}(iii) this implies that $(f,g,h,k)$ restricts to a homotopy equivalence between a circle and a point. Again, this leads to the required contradiction.
    \end{itemize}
    
\end{proof}

\end{document}